\documentclass[a4paper,10pt]{amsart}
\usepackage{amssymb,latexsym,amsmath,amsfonts, amsthm}
\usepackage{latexsym}
\usepackage{color}
\usepackage[mathscr]{eucal}
\usepackage{comment}
\usepackage{dsfont}
\newcommand{\mc}[1]{\mathcal{#1}}
\usepackage{cancel}

\usepackage[linkcolor=blue, citecolor=red]{hyperref}
\hypersetup{colorlinks=true, urlcolor=blue}
\usepackage[nameinlink]{cleveref}

\usepackage[sort,nocompress]{cite}

\usepackage[top=2.85cm,bottom=2.75cm,left=2.7cm,right=2.7cm]{geometry}

\numberwithin{equation}{section}

\theoremstyle{definition}
\newtheorem{definition}{Definition}[section]

\theoremstyle{remark}
\newtheorem{remark}[definition]{Remark}
\theoremstyle{plain}
\newtheorem{theorem}[definition]{Theorem}
\newtheorem{lemma}[definition]{Lemma}
\newtheorem{proposition}[definition]{Proposition}

\newcommand{\norm}[1]{\left\|#1\right\|}
\newcommand{\abs}[1]{\left|#1\right|}

\newcommand{\ip}[2]{\left<{#1},{#2}\right>}

\newcommand{\br}{\overline}

\newcommand{\vphi}{\varphi}

\newcommand{\cplx}{\mathbb{C}}

\newcommand{\rea}{\mathbb{R}}

\newcommand{\iintq}{\iint_{Q_T}}

\newcommand{\re}{\textnormal{Re}}

\newcommand{\im}{\textnormal{Im}}

\makeatletter
\@namedef{subjclassname@2020}{%
	\textup{2020} Mathematics Subject Classification}
\makeatother

\title[Control and stabilization issues for coupled Schr\"odinger system]{Coupled linear Schr\"odinger equations: Control and stabilization results}

\author[Bhandari, Capistrano-Filho, Majumdar, Tanaka]{K. Bhandari$\,^{\dagger}$ \and R. de A. Capistrano-Filho$\,^\ddagger \, ^\star$ \and S. Majumdar$\,^\S$ \and  T. Y. Tanaka$\,^\P$}
\thanks{$^\star$Corresponding author: roberto.capistranofilho@ufpe.br\\
$^{\dagger}$Institute of Mathematics of the  Czech Academy of Sciences, \v{Z}itn\'a 25, 11567 Praha 1, Czech Republic; \\ bhandari@math.cas.cz \\
	$^{\ddagger}$Departamento de Matem\'atica, Universidade Federal de Pernambuco (UFPE), 50740-545, Recife (PE), Brazil;
	roberto.capistranofilho@ufpe.br \\
$^\S$Department of Mathematics, Indian Institute of Technology Bombay, Powai, Mumbai 400076, India; \\
	subratamajumdar634@gmail.com\\
$^\P$Departamento de Matem\'atica, Universidade Federal Rural de Pernambuco (UFRPE), 52171-900, Recife (PE), Brazil;
	thiago.tanaka@ufrpe.br}

\date{\today}

\begin{document}
	
	\begin{abstract} This article presents some controllability and stabilization results for a system of two coupled linear Schr\"odinger equations in the one-dimensional case where the state components are interacting through the  Kirchhoff boundary conditions.  Considering the system in a bounded domain, the null boundary controllability result is shown. The result is achieved thanks to a new Carleman estimate, which ensures a boundary observation. Additionally, this boundary observation together with some trace estimates, helps us to use the Gramian approach, with a suitable choice of feedback law, to prove that the system under consideration decays exponentially to zero at least as fast as the function $e^{-2\omega t}$ for some $\omega>0$.  
	\end{abstract}
	\maketitle
	
	\section{Introduction}
	This work is dedicated to the study of the boundary control problem and stabilization issue of a linear system that appears to model some problems in the context of nonlinear optics. Precisely, our motivation comes from the following system 
	\begin{equation}\label{nonlinear}
		\begin{cases}
i u_t(t,x)+p u_{xx}(t,x)-\theta u(t,x)+\bar{u}(t,x)v(t,x)=0, \quad t \geq 0, \ x \in \mathbb{R}, \\
i \sigma v_t(t,x)+q  v_{xx}(t,x)-\varrho v(t,x)+\frac{a}{2} u^2(t,x)=0, \\
u(0,x)=u_0(x), \quad v(0,x)=v_0(x),
		\end{cases} 
	\end{equation}
where $u$ and $v$ are complex-valued functions and $\theta$, $\varrho$ and $a:=1 / \sigma$ are real numbers representing physical parameters of the system, where $\sigma>0$ and $p, q= \pm 1$. Notice that the system \eqref{nonlinear} is given by the nonlinear coupling of two dispersive equations of Schrödinger type through the quadratic terms.
	
	There are some physical meanings for the previous system, as mentioned before. For example, the complex functions $u$ and $v$ represent amplitude packets of an optical wave's first and second harmonic, respectively. The values of $p$ and $q$ depend on the signals provided between the scattering/diffraction ratios, and the positive constant $\sigma$ measures the scaling/diffraction magnitude indices. For details about this system, the authors suggested the references \cite{MeChTo,DeSalvo,KaSu}, and the references therein. 
	
		Concerning the mathematical context, most of the work related to the system \eqref{nonlinear} is devoted to proving the well-posedness of the Cauchy problem in $\mathbb{R}^n$, for $n\in[1,6]$ or in a periodic framework $\mathbb{T}$.  For example, in \cite{HaOzTa} the authors showed the local well-posedness for the Cauchy problem \eqref{nonlinear} on the spaces $L^2\left(\mathbb{R}^n\right) \times L^2\left(\mathbb{R}^n\right)$ for $n \leq 4$ and $H^1\left(\mathbb{R}^n\right)$ $\times H^1\left(\mathbb{R}^n\right)$ for $n \leq 6$.  About qualitative properties of Cauchy problem \eqref{nonlinear}, the case where $p=q=1$ was studied in \cite{AnLi} for initial data $u_0, v_0$ in the same periodic Sobolev space $H^s(\mathbb{T})$. We also cite that in \cite{barbosa}, the author studied the well-posedness of the Cauchy problem associated with the coupled Schr\"odinger equations with quadratic nonlinearities. He obtained the local well-posedness for data in Sobolev spaces with low regularity. Finally, the authors suggest the reference \cite{LiHa} for the recent progress on nonlinear Schr\"odinger systems with quadratic interactions. 
		
		In the context of the control theory, no author attempted to show controllability results in bounded domains for the system \eqref{nonlinear}. Due to this fact, our motivation is to present the boundary control results for the linear system associated with \eqref{nonlinear} posed in a bounded domain, giving a necessary first step in the direction to prove the nonlinear results for the system \eqref{nonlinear}. 
		
\subsection{Problem setting}
As mentioned before, our motivation in this work is to present, as a first step, the control results to the linear Schr\"odinger system associated with \eqref{nonlinear}. More precisely, considering $T>0$ be any finite time and $\Omega=(0,1)$, we define $Q_T: = (0,T)\times \Omega$ and $\Sigma_T := (0,T)\times \partial \Omega$. So, we will study the boundary controllability of the following linearized system
	\begin{equation}\label{linear}
		\begin{cases}
			iu_t+  \gamma_1 u_{xx}- \alpha_1 u  = 0 &\text{ in } Q_T,\\
			i\sigma v_t + \gamma_2 v_{xx}-\alpha_2 v = 0 & \text{ in } Q_T,\\
			u(0,x)=u_0(x), \ v(0,x)=v_0(x) & \text{ in } \Omega,
		\end{cases} 
	\end{equation}
	where the constants $\sigma,\gamma_1, \gamma_2, \alpha_1, \alpha_2$ are positive and $(u_0,v_0)$ are given initial data in certain spaces which will be specified later.  We will consider the system \eqref{linear} with the so-called \textit{Kirchhoff boundary condition} at the right spatial point $x=1$:
\begin{equation}\label{boundary-1}
		\begin{cases}
			u(t,1)=v(t,1)   \ & \text{in } (0,T), \\
			\gamma_1 u_x(t,1)+ \frac{\gamma_2}{\sigma} v_x(t,1)+\alpha u(t,1)=0 \ & \text{in } (0,T),
		\end{cases}
	\end{equation}
	with positive parameter $\alpha$. 
	 
	Here, the boundary control $h\in L^2(0,T)$ acts 
	 either on the component $u$ or  $v$ at the left spatial point $x=0$; to be more precise we set
	\begin{subequations} 
		\begin{align} \label{control-u}
			\text{either } \ \ u(t,0) = h(t) , \ \ v(t,0) =0 \quad & \text{ in } (0,T) , \\
			\label{control-v} 
	\text{or } \ \   u(t,0) = 0 , \ \ v(t,0) = h(t) \quad & \text{ in } (0,T) .
		\end{align}
	\end{subequations}
	So, the first goal of this article is to answer the following null controllability  problem: 
		
	\vspace{0.2cm}
	\noindent\textbf{Problem $\mathcal{A}$:} Given $T>0$, $(u_0,v_0)$ in a certain space ${X}$, can one find an appropriate control input $h$ such that the corresponding solution $(u,v)$ of \eqref{linear} with boundary conditions \eqref{boundary-1}-\eqref{control-v} (or \eqref{boundary-1}-\eqref{control-u}) satisfies
\begin{equation}\label{null}
(u(0,x),v(0,x))=(u_0(x),v_0(x)) \quad \text{and} \quad (u( T,x),v(T,x))=(0,0),  \ \ \forall x \in \Omega \  ?
\end{equation}
	
 If we can positively answer the previous question, an interesting problem is to study the boundary stabilization problem for the system  \eqref{linear} with boundary conditions \eqref{boundary-1} and with a single boundary control force exerted on the component $v$ (or $u$), namely \eqref{control-v} (or \eqref{control-u}). In this context, the second main problem in this work treats the following stabilization issue: 
	
	\vspace{0.2cm}
	\noindent\textbf{Problem $\mathcal{B}$:} Can we construct a stationary feedback law $h(t)$, of the form ${F}_{\omega}(u(t,\cdot), v(t,\cdot))$, such that the solution of the closed-loop system \eqref{linear} with boundary condition \eqref{boundary-1}-\eqref{control-v} (or \eqref{boundary-1}-\eqref{control-u}) decays exponentially to zero at any prescribed decay rate?
		\vspace{0.2cm}
		
		An important point to answer these questions is that we need to assume the parameter $\sigma$, in the second equation of \eqref{linear}, be a positive number and the choice of $\sigma$ will play a crucial role in deducing the required controllability result, which will be discussed further up. Moreover, we mention that now, the second problem will be called the \textit{rapid stabilization problem}.
		
		Before giving details about the main results of the article and some important facts, let us give a brief history of the control problem for the Schrödinger type systems, as well as some references to the rapid stabilization issues of partial differential equations (PDEs). 
		
		\subsection{Literature review}  We are not aware of any results for systems where the coupling is given at the Kirchhoff boundary condition for the Schr\"odinger type systems, as is our case. However, concerning the coupled (internal) structure in the equation, there are several results considering the cascade system for Schr\"odinger equations, which we would like to mention. 

\vspace{0.1cm}

 We warn ourselves that this is only a small sample of the results concerning the control, stabilization problems, and some methods. We strongly encourage the reader to see the references cited above as well as the references therein for more details about all of these problems. 
				
	\subsubsection{Control results for Schr\"odinger equation}It is well known that control properties for a single Schr\"{o}dinger equation have received a lot of attention in the last decades (see, e.g., \cite{Zuazua,Laurent} for an excellent review of the contributions up to 2014).  There is an ongoing effort to show new control results for this equation, and this effort is giving significant progress for control properties. So, in this spirit, we can cite \cite{KoLo,Miller,Phung,RaTaTeTu,RoZhaSIAM,RoZhaMMM} and the references therein for control issues or \cite{BaPu,CaGa,LaTriZhang,MeOsRo,YuYa} and the references therein for Carleman estimates and their applications to inverse problems. 

\subsubsection{Control results for coupled system} We are not aware of any results for systems where the coupling is given at the Kirchhoff boundary condition for the Schr\"odinger type systems, as is our case. However, concerning the coupled (internal) structure in the equation, there are several results considering the cascade system for Schr\"odinger equations, which we would like to mention.  We advertise that this is only a small sample of the results concerning the control and stabilization problem. We strongly encourage the reader to see the previous references and the references therein for more details about all of these problems. 

We first quote that Fatiha's articles \cite{AB2001,AB2003} are the first to establish observability and controllability results for coupled systems with less than the number of coupled equations, moreover, we mention that in these articles the author deals with symmetrically coupled systems such as the system \eqref{linear}. Moreover, the articles, \cite{AB2012-2,FaLe,AB2015} deal with cascade systems and also with coupled parabolic or diffusive systems (including Schrödinger coupled systems) as well as \cite{ABL2011}. We infer to the reader to see also the reference therein.  Additionally, we mention that controllability results for systems of parabolic equations are reviewed \cite{AmBeGonTe}.

 Concerning systems of hyperbolic equations, we cite the works \cite{Fa,DeRoLe,ABCO2017}. There and the references therein, the reader can find results about the controllability of two coupled wave equations with only one control, under the hypothesis of the geometric control condition and results of exact controllability of $n \times n$ first order one-dimensional quasi-linear hyperbolic systems by $m < n$ internal controls that are localized in space in some part of the domain.

More recently, in \cite{RoTe} a boundary controllability result is shown for a Schr\"odinger cascade type system with periodic boundary conditions. This result is obtained as a consequence of the controllability result for a cascade system of two wave equations.  We also refer to the work \cite{LoMeTe}, where the authors studied the null controllability of a linear system formed by two Schr\"odinger, controlling only one of them using Carleman estimates. Also, pay attention to the survey \cite{AB2012-1} where stabilization, observability, and control of coupled systems with a reduced number of controls are presented.

Finally, the first author showed in a recent work \cite{bhandari2021boundary} the boundary controllability of some $2 \times 2$ one-dimensional parabolic systems with both the interior and boundary couplings: The interior coupling is chosen to be linear with a constant coefficient while the boundary one is considered through some Kirchhoff-type condition at one end of the domain.  Additionally, considering two uncoupled wave equations with potentials on an interval, in \cite{RoTeb}, the authors established a Carleman estimate for wave systems with simultaneous boundary control, giving a boundary controllability result for uncoupled wave equations.

\subsubsection{Rapid stabilization of PDEs}  For rapid stabilization issues, in recent years, some abstract methods \cite{Kom-1997,Urquiza,Vest} have been developed to obtain answers considering linear PDEs. The method is based on the Gramian approach and the Riccati equations, and several authors employed during the last few years this approach. For example, we can cite, \cite{CeCr} for the KdV equation in a bounded domain, \cite{CaCeGa} for the KdV-KdV equation with only one boundary feedback acting, \cite{JaKo} for one-dimensional Schrödinger equation and of the beam and plate equations by moving or oblique observations and, additionally, for vibrating strings and beams, we can refer to \cite{Bar-09}.

\subsection{Main results and further comments} We are now in a position to give comments on our main contribution to this article. 
 Consider, now on, the space 
\begin{align}\label{space-H} 
	\mathcal{H}=\big\{(u_1, u_2)\in [H^1(\Omega)]^2 \mid  u_1(0)=u_2(0)=0, \ u_1(1)=u_2(1) \big\},
\end{align}
as the natural space for belonging of the initial data associated with \eqref{linear} with the following norm
\begin{align}\label{def-norm-H} 
	\norm{(u_1,u_2)}_{\mc{H}}=\bigg(\int_\Omega \left( |u_1^\prime(x)|^2 + |u_2^\prime(x)|^2\right) dx\bigg)^{\frac{1}{2}} ,
\end{align}   
and the associated inner product defined by 
\begin{equation*}
	\ip{(u_1,u_2)}{(v_1,v_2)}_{\mc{H}}=\re \int_\Omega u_1^\prime(x)\br{v_1^\prime(x)} dx  + \re \int_\Omega u_2^\prime(x)\br{v_2^\prime(x)} dx,
\end{equation*}
for any $(u_1, u_2), \, (v_1, v_2)\in \mc{H}$. Finally, we denote $\mc{H}^\prime$ as the dual space of $\mc{H}$ with respect to the pivot space  $[L^2(\Omega)]^2$. 

The first result of our work gives the control problem for the system \eqref{linear} with boundary conditions \eqref{boundary-1}-\eqref{control-v}, that is when the control is acting on the second component. Precisely, considering these boundary conditions, due to a new Carleman estimate with boundary observation, the following result is verified. 

\begin{theorem}\label{th-main} Let the set
\begin{align}\label{Set-sigma}
	\mathfrak{S}:  = \left\{  \sigma >0 \mid  \sigma = \frac{\kappa \gamma_2}{\gamma_1}, \ \ \kappa>3      \right\} ,
\end{align}
where $\gamma_1, \gamma_2>0$ are as appearing in \eqref{linear}.  For any $T>0$, initial data $\left(u_0, v_0\right) \in \mathcal{H}'$ and parameters $\gamma_1, \gamma_2, \alpha_1, \alpha_2, \alpha$, and for any $\sigma \in \mathfrak{S}$,  there exists a  control $h \in L^2(0,T)$ such that the solution $(u, v)$ to the system \eqref{linear} with boundary conditions \eqref{boundary-1}-\eqref{control-v}
 satisfies \eqref{null}.
\end{theorem}

Since we need to prove an observability inequality (for the associated adjoint system) to give the proof of the previous theorem, naturally, the Problem $\mathcal{B}$ seems reachable. The next result gives, for the coupled Schr\"odinger equation \eqref{linear}, the following positive answer for the rapid stabilization problem. 

\begin{theorem}\label{th-main-1}
	 Let any 
	  parameters $\gamma_1, \gamma_2, \alpha_1, \alpha_2, \alpha$ be given positive.
	Then, for the same choices of $\sigma$ as in Theorem \ref{th-main},  there exists a continuous linear map $ {F}_{\omega}: \mathcal{H} \to \mathbb{C}$ and a positive constants $C$ and $\omega$, such that for every $\left(u_0, v_0\right) \in \mathcal{H}$, the solution $(u, v)$ of the closed-loop system \eqref{linear} with boundary conditions \eqref{boundary-1}-\eqref{control-v}, with $h(t)={F}_{\omega}(u(t,\cdot), v(t,\cdot))$ satisfies
\begin{equation}\label{ex_p}
\|(u(t), v(t))\|_{\mathcal{H}} \leq C e^{-2 \omega t}\left\|\left(u_0, v_0\right)\right\|_{\mathcal{H}}, \quad \forall t \geq 0.
\end{equation}
\end{theorem}

\begin{remark} In what concerns our main results, Theorems \ref{th-main} and \ref{th-main-1}, the following remarks are worth mentioning:
\begin{itemize} 
\item As usual in the literature the answer for the Problem $\mathcal{A}$, that is, Theorem \ref{th-main}, is shown by using the Hilbert Uniqueness Method introduced by Lions \cite{Lions} and the classical duality theory of Dolecki and Russell \cite{DolRus-1977}. For that, it is essential to prove a suitable observability inequality with boundary observation, and to do so, in the present article, we prove a new Carleman estimate for the associated adjoint system to \eqref{linear}-\eqref{boundary-1}-\eqref{control-v}.  
\vspace{0.1cm}
\item An important fact is that our approach, that is, to design a feedback law to stabilize the closed-loop system \eqref{linear} with boundary conditions \eqref{boundary-1}-\eqref{control-v}, was introduced by Lukes \cite{Lukes} and Kleinman
\cite{Kleinman} for the finite-dimensional systems. Later on, Slemrod \cite{Slemrod}, adapted the result to improve the stabilization for infinite-dimensional systems with bounded control operators.
\vspace{0.1cm}
\item We study the control and stabilization problems for the system \eqref{linear} with boundary conditions \eqref{boundary-1}-\eqref{control-v}, however, we point out that a similar analysis can be performed if we instead consider \eqref{boundary-1}-\eqref{control-u}.
\vspace{0.1cm} 
\item  As mentioned before, in \cite{bhandari2021boundary} the authors showed the boundary null-controllability properties for 1-D coupled parabolic systems, there are differences between our work and the previous one. The first one is that, in our case, the choice of such $\sigma\in\mathfrak{S}$ in Theorem \ref{th-main} is important to deduce the required controllability result via Carleman estimate;  for more details, we refer to Remark \ref{Remark-sigma}.  Another interesting point is that we can employ the classical Gramian approach \cite{Kom-1997,Urquiza,Vest}, giving the proof of Theorem \ref{th-main-1}, and consequently, answering the Problem $\mathcal{B}$, which was not achieved in this work.
\vspace{0.1cm}
\item We mention that we are interested in giving control results for the system \eqref{nonlinear}, however, due to the structure of the nonlinearities the well-posedness problem is an open issue. So, in our case, we just give the necessary first step to understand this system.
\end{itemize}
\end{remark}

\subsection{Structure of the article} Our work is composed of five parts, including the introduction. In Section \ref{sec2}, the boundary controllability is considered. We obtain a new Carleman estimate with boundary observation, which is the key point to prove the observability inequality. So, with this observability in hand, the Theorem \ref{th-main} is verified. In Section \ref{sec3}, we recall Urquiza's approach and use it to achieve the second main result of the article, i.e.,  Theorem \ref{th-main-1}. Section \ref{sec4} is devoted to presenting some comments and open issues. Finally, in Appendix \ref{apx}, we present an overview of the well-posedness results, for the direct and adjoint systems associated with \eqref{linear}. Additionally, a key lemma, essential to prove the rapid stabilization result, is proved in Appendix \ref{apx2}.


	\section{Boundary controllability}\label{sec2} 
	
	Let us first study the global null-controllability properties of the system \eqref{linear}-\eqref{boundary-1} when the control acts on the component $v$, that is precisely \eqref{control-v}. The main tool is to establish a suitable Carleman estimate that yields a proper observability inequality and utilizing that, we prove the required controllability result for the concerned model. 
	
		 \begin{definition}
For a given $T > 0$, the system \eqref{linear} with boundary condition \eqref{boundary-1}-\eqref{control-v} is null controllable at time $T$ if for any given initial data $(u_0,v_0)\in \mathcal{H}'$, there exists  a control function $h\in L^2(0,T)$, such that solution  $(u,v)$ to \eqref{linear} satisfies \eqref{null}.
\end{definition}

Additionally, the solution by transposition of \eqref{linear}-\eqref{boundary-1}-\eqref{control-v} is given below. 
	\begin{definition}Let $(u_0, v_0)\in \mc{H}'$ and $h\in L^2(0,T)$. We say that $(u,v)\in L^{\infty}(0,T;\mc{H}')$ is solution of 
		\begin{align}\label{linearsource}
			\begin{cases}
				iu_t+  \gamma_1 u_{xx}- \alpha_1 u  =0 &\text{ in } Q_T,\\
				i\sigma v_t + \gamma_2 v_{xx}-\alpha_2 v= 0 & \text{ in } Q_T,\\
				u(t,0) =0 & \text{ in } (0,T),\\
				v(t,0) = h(t) & \text{ in } (0,T),\\
				u(t,1)=v(t,1)   \ & \text{ in } (0,T), \\
				\gamma_1 u_x(t,1)+ \frac{\gamma_2}{\sigma} v_x(t,1) + \alpha u(t,1) =0  & \text{ in } (0,T),\\
				u(0,x)=u_0(x), \ v(0,x)=v_0(x) & \text{ in } \Omega,
			\end{cases} 
		\end{align}
in the transposition sense if and only if 
		\begin{equation}\label{def}
			\begin{split}
			\int_{0}^{T}\left\langle (u(t),v(t)),(g_1(t), g_2(t)) \right\rangle_{\mathcal{H}',\mathcal{H}}dt =& \left\langle (u_0,v_0),({\varphi_1(0,\cdot)}, {\vphi_2(0,\cdot)})\right\rangle_{\mathcal{H}',\mathcal{H}}
				\\&- \re\int_{0}^{T} h(t)\br{\varphi_{2,x}(t,0)}\, dt, 
			\end{split}
		\end{equation}
for every $(g_1,g_2)\in L^1(0,T; \mc H)$ where $( \vphi_1,\vphi_2)$ are the mild solution to the problem
		\begin{equation}	\label{adjoint tr}
			\begin{cases}
				i\vphi_{1,t}+ \gamma_1\vphi_{1,xx} -\alpha_1\vphi_1 = g_1, &\text{ in } Q_T,\\
				i\sigma \vphi_{2,t} +  \gamma_2\vphi_{2,xx}-\alpha_2 \vphi_{2}= g_2, & \text{ in } Q_T,\\
				\vphi_1(t,1)=\vphi_2(t,1), & \text{ in } (0,T),\\
				\gamma_1\vphi_{1,x}(t,1)+\frac{\gamma_2}{\sigma} \vphi_{2,x}(t,1)+\alpha \vphi_1(t,1)=0, & \text{ in } (0,T),\\
				\vphi_1(t,0)=0,\, \vphi_2(t,0)=0, & \text{ in } (0,T),
			\end{cases} 
		\end{equation}
		 on the space $ C([0,T];\mc H)$, with $(\vphi_{1}(T,\cdot),\vphi_{2}(T,\cdot))=(0,0)$ in $\Omega$.
	\end{definition}

We remark that the discussion of the system \eqref{adjoint tr} is given in Appendix \ref{apx}. Moreover, the well-posedness of the control system \eqref{linearsource} is given by the following proposition.
\begin{proposition}\label{wl th}
	Let $(u_0, v_0)\in \mc{H}'$ and $h\in L^2(0,T)$. Then the control system \eqref{linearsource} has a unique solution $(u,v)$ in $C([0,T]; \mc H').$
\end{proposition}
\begin{proof}
The well-posedness of the control system \eqref{linearsource} is a consequence of the property of \textit{admissibility of the control operator}, namely, 
$$
\int_{0}^{T}|v_{x}(t,0)|^2dt \leq C \norm{(u_0, v_0)}_{\mathcal{H'}},
$$
as observed, for example, by Lasiecka and Triggiani in \cite{Irena}. Therefore, since the proof of the previous inequality will be given in Proposition \ref{irena1}, for the solutions of the adjoint system \eqref{adjoint tr}, we will omit the details of the proof here. An explanation about this point can be seen in Remark \ref{irena}.
\end{proof}

		\subsection{Global Carleman estimate}
		
		With the previous definitions, we are in a position to obtain the global Carleman estimate for the adjoint system associated with the system \eqref{linear} with boundary conditions \eqref{boundary-1}-\eqref{control-v}, namely, 
			\begin{equation}\label{adjoint}
		\begin{cases}
			i\vphi_{1,t}+ \gamma_1\vphi_{1,xx} -\alpha_1\vphi_1 = 0, &\text{ in } Q_T,\\
			i\sigma \vphi_{2,t} +  \gamma_2\vphi_{2,xx}-\alpha_2 \vphi_{2}= 0, & \text{ in } Q_T,\\
			\vphi_1(t,1)=\vphi_2(t,1), & \text{ in } (0,T),\\
			\gamma_1\vphi_{1,x}(t,1)+\frac{\gamma_2}{\sigma} \vphi_{2,x}(t,1)+ \alpha \vphi_1(t,1)=0, & \text{ in } (0,T),\\
			\vphi_1(t,0)=0,\, \vphi_2(t,0)=0, & \text{ in } (0,T),\\
			\vphi_{1}(T,x)=\zeta_1(x), \  \vphi_{2}(T,x)=\zeta_2(x), & \text{ in } \Omega,
		\end{cases} 
	\end{equation}
	with $\zeta:=(\zeta_1, \zeta_2)\in \mathcal H$. 
	
	To this end, we introduce the space 
		\begin{equation*}
		\begin{split}
		\mathcal Q : = &\Big\{(\vphi_1, \vphi_2) \in [C^2(\overline{Q_T})]^2 \mid \vphi_1(t,0)    = \vphi_2(t,0) = 0, \ \vphi_1(t,1) = \vphi_2(t,1) , \\ &\quad
		\qquad \qquad \qquad  \gamma_1\vphi_{1,x}(t,1)+\frac{\gamma_2}{\sigma} \vphi_{2,x}(t,1)+\alpha \vphi_1(t,1)=0 , \ \forall t\in [0,T]  \Big\}	. 
		\end{split}
	\end{equation*}

Now recall that $\sigma\in \mathfrak{S}$ (see \eqref{Set-sigma}) and thus 
\begin{align}\label{Value-sigma} 
\sigma= \frac{\kappa \gamma_2}{\gamma_1} \ \text{ for some $\kappa>3$}.
\end{align}
 With this in hand,   we consider  the  following auxiliary functions (motivated from \cite{bhandari2021boundary}):
	\begin{align}\label{aux-func}
		\begin{cases} 
			\beta_{j}(x)=2+c_j(x-1),  \ \ \ j=1,2, \\
			c_1=1, \  c_2= - \kappa \, \text{ and thus } |c_2|>3  .
		\end{cases}
	\end{align}
	Therefore, $\beta_j \in C^2([0,1])$ and satisfy
	$$
		\beta_2\geq \beta_1 >0, \ \ \text{in } \, [0,1], \ \ \beta_1(1)=\beta_2(1)  .
$$
	
	Next, for any parameter $\lambda>1$, we  introduce the following   weight functions:
	\begin{align}\label{weight-func}
		\xi_j(t,x)=\frac{e^{\lambda \beta_j(x)}}{t(T-t)}, \quad \eta_j(t,x)=\frac{e^{2\lambda\norm{\beta_j}_{\infty}}-e^{\lambda \beta_j(x)}}{t(T-t)} , \quad \forall (t,x) \in Q_T, \ \ j=1,2. 
	\end{align}
	Note that 
$$
		\xi_j, \, \eta_j >0 \, \text{ for }j=1,2 ,\ \  	\xi_1(t,1)=\xi_2(t,1), \ \ \eta_1(t,1) = \eta_2(t,1) ,
$$
	since $\beta_1(1)=\beta_2(1)$. Concerning of the function $\xi_j$ and $\eta_j$, we have the following behavior in $Q_T$: 
	\begin{equation}\label{derivatives-weights}
		\begin{cases}
			\xi_{j,x} = \lambda \xi_j c_j , \ \ \eta_{j,x} = - \lambda \xi_j c_j , \\ \eta_{j,xx} = -\lambda^2 c_j^2 \xi_j , \ \ \eta_{j,xxx} = -\lambda^3 c_j^3 \xi_j   ,  \\
			|\xi_{j,t}| \leq C T \xi^2_j , \ \ |\eta_{j,t}| \leq CT \xi^2_j , \\
			|\eta_{j, xt}| \leq C\lambda  T \xi^2_j , \ \ |\eta_{j, tt}| \leq C T^2 \xi^3_j    , 
		\end{cases}
		\text{for $j=1,2$.}
	\end{equation}

	\begin{remark}\label{Remark-sigma}
	  We point out that the choice of $c_2$ in \eqref{aux-func} and the value of $\sigma \in \mathfrak{S}$  in \eqref{Set-sigma} are crucial to obtain the Carleman estimate \eqref{finalcarlemanestimate}. More precisely,  those choices help us deal with some unusual boundary integrals while proving the underlying Carleman estimate. Here, we must mention that the
	  special choice of parameter $\sigma$ does not occur 
	  in \cite{bhandari2021boundary} (roughly because there were no complex parts in the Carleman estimate) which differs the both Carleman estimates shown in the present work and in \cite{bhandari2021boundary}.
	\end{remark}
	
The main theorem of this subsection can be read as follows:
	\begin{theorem}[Carleman estimate]\label{Thm.1}
		Let the weight functions $\xi_1, \xi_2$, $\eta_1,\eta_2$ be chosen as in \eqref{weight-func} and $\sigma>0$ is taken as   \eqref{Value-sigma}.  Then,
		there exist constants  $\lambda_0>0$, $\mu_0>0$ and $C>0$, depending at most on $\gamma_1, \gamma_2$ and $c_2$, such that the following  estimate holds true
		\begin{equation}\label{finalcarlemanestimate}
			\begin{split}
				s^3&\lambda^4 \iintq \left(e^{-2s\eta_1}\xi^3_1 |\vphi_1|^2 + e^{-2s\eta_2}\xi^3_2 |\vphi_2|^2 \right)dxdt \\&+ s\lambda^2 \iintq \left(e^{-2s\eta_1}\xi_1 |\vphi_{1,x}|^2 + e^{-2s\eta_2}\xi_2 |\vphi_{2,x}|^2 \right)dxdt   \\  
		\leq& C \iintq (e^{-2s\eta_1}|L_1\vphi_1|^2 + e^{-2s\eta_2}|L_2\vphi_2|^2)dxdt \\& + C s\lambda \int_0^T e^{-2s\eta_2(t,0)} \xi_2(t,0) |\vphi_{2,x}(t,0)|^2dt, 
			\end{split}
		\end{equation}
		for every $\lambda\geq \lambda_0$,  $s\geq s_0:=\mu_0(T+T^2)$  and for all $(\vphi_1, \vphi_2)\in \mathcal{Q}$, where 
		\begin{equation*}
		L_1=i\partial_t + \gamma_1 \partial_{xx} \quad \text{and} \quad
		L_2=i\partial_t+\frac{\gamma_2}{\sigma} \partial_{xx}.
				\end{equation*} 
	\end{theorem}
	
	To prove the above theorem, let us now  define the following variables:
	\begin{align}\label{variables-psi der}
		\psi_j(t,x)=e^{-s\eta_j(t,x)}\vphi_j(t,x), \quad   \forall (t,x)  \in Q_T, \ \ \text{for } j=1,2,
	\end{align}
	so that we have the following boundary conditions:
	\begin{align}\label{boundary-psi}
		\begin{cases}
			\psi_j(t,0) = 0 , \quad \psi_1(t,1) = \psi_2(t,1) , \\
			\gamma_1\psi_{1,x}(t,1) + \frac{\gamma_2}{\sigma} \psi_{2,x}(t,1) =-\alpha\psi_1(t,1)+ s\lambda \xi_1(t,1) \psi_1(t,1) \left(\gamma_1 c_1 + \frac{\gamma_2}{\sigma} c_2 \right). 
		\end{cases}
	\end{align}

	We further denote 
$$
		F_j: = e^{-s\eta_j} L_j(\vphi_j)=e^{-s\eta_j}L_j(e^{s\eta_j}\psi_j),
$$
	where $L_1$ and $L_2$. So, with this, we find  the following relations  for each $j=1,2$,
	\begin{align*}
		&i(\psi_j e^{s\eta_j})_{t}=i \psi_{j,t}e^{s\eta_j}+i s \psi_j  e^{s\eta_j}\eta_{j,t},\\
		&(\psi_j e^{s\eta_j})_x= \psi_{j,x} e^{s\eta_j} + s \psi_j e^{s\eta_j}\eta_{j,x},\\
		&(\psi_j e^{s\eta_j})_{xx}= \psi_{j,xx} e^{s\eta_j} + 2se^{s\eta_j}  \psi_{j,x} \eta_{j,x}  + s^2\psi_j e^{s\eta_j} |\eta_{j,x}|^2 + s\psi_j e^{s\eta_j} \eta_{j,xx} .
	\end{align*}
Using the previous relation, $F_j$ can be written as follows
	\begin{align}\label{equation-F-j}
		M_1 \psi_j + M_2 \psi_j = F_j ,
	\end{align} 
	where 
		\begin{align*}
		\begin{cases}
			M_1 \psi_1  = 2s \gamma_1  \psi_{1,x}  \eta_{1,x} + s \gamma_1 \psi_1 \eta_{1,xx}  +  i s \psi_1 \eta_{1,t} , \\
			M_2 \psi_1  =i \psi_{1,t}+ \gamma_1 \psi_{1, xx} +  s^2 \gamma_1  |\eta_{1,x}|^2 \psi_1, 
		\end{cases}
	\end{align*}
	and
	\begin{align}\label{exp-M-psi-2}
		\begin{cases}
			M_1 \psi_2  = 2s \frac{\gamma_2}{\sigma}  \psi_{2,x}  \eta_{2,x} + s \frac{\gamma_2}{\sigma} \psi_2 \eta_{2,xx}  +  i s  \psi_2 \eta_{2,t} , \\
			M_2 \psi_2  =i  \psi_{2,t}+ \frac{\gamma_2}{\sigma} \psi_{2, xx} +  s^2 \frac{\gamma_2}{\sigma}  |\eta_{2,x}|^2 \psi_2 .
		\end{cases}
	\end{align}

	Thus, we get from \eqref{equation-F-j}, 
	\begin{equation}\label{equ-inner-product}
	\begin{split}
		\iintq \left(|M_1 \psi_j|^2 + |M_2 \psi_j|^2\right)dxdt &+ 2 \re  \iintq M_1 \psi_j \overline{M_2 \psi_j}dxdt \\&= \iintq |F_j|^2dxdt,
		\end{split}
	\end{equation}
		for $j=1,2$. Now, we are in a position to prove Theorem \ref{Thm.1}.
	
	\begin{proof}[Proof of Theorem \ref{Thm.1}] Our goal is to focus on the following inner product $$ \re  \iintq M_1 \psi_j \overline{M_2 \psi_j}dxdt$$ that contains $9$ terms.  We will elaborately make the computations for $j=2$, similarly, the computations can be done for $j=1$ as well. 
				
				Recall that the quantities $M_1\psi_2$ and $M_2\psi_2$ are given by \eqref{exp-M-psi-2} and, for $j=2$, we have the relation \eqref{equ-inner-product}. We further denote  
$$
			\re  \iintq M_1 \psi_2 \overline{M_2 \psi_2}dxdt = \sum_{1\leq k,l\leq 3} I_{kl},
$$
		where all the terms  $I_{kl}$ for $1\leq k,l \leq 3$ consists of the integral term with the product involving the $k$-th term of $M_1 \psi_2$ with the $l$-th term of $M_2\psi_2$, and will be computed below.  Now, we split the proof into several steps. 
		
		\vspace{0.2cm}
		
		\noindent  $\bullet$ {\bf Step 1.}  Computations of the terms $I_{11}$, $I_{21}$ and $I_{32}$. 

		\vspace{0.2cm}
		
Let us start with $I_{11}$. Observe that
		\begin{equation}\label{term-I_11}
		\begin{split}
			I_{11} =& 2s \frac{\gamma_2}{\sigma} \re \iintq \psi_{2,x} \eta_{2,x} \overline{i \psi_{2,t}}dxdt \\=& 2s \frac{\gamma_2}{\sigma} \im \iintq  \psi_{2,x} \eta_{2,x} \overline{\psi_{2,t}} dxdt \\
			 =& - 2s \frac{\gamma_2}{\sigma} \im \iintq  \psi_{2,x} \eta_{2,xt} \overline{\psi_{2}}dxdt - 2s\frac{\gamma_2}{\sigma}\im \iintq  \psi_{2,xt} \eta_{2,x} \overline{\psi_{2}} dxdt \\
			 :=&J_1 + J_2 ,
			 \end{split}
		\end{equation}
		where there is no boundary integral since $$\lim\limits_{t\to 0^+}\psi_2(t,\cdot)= \lim\limits_{t\to T^-} \psi_2(t,\cdot)=0,$$ thanks to the choices of weight functions \eqref{weight-func}. 
			
			\vspace{0.2cm}
			
		Next, for the quantity $I_{21}$, we see that 
		\begin{equation}\label{term-I-21-1}
			\begin{split}
					I_{21}= &
				s  \frac{\gamma_2}{\sigma} \re \iintq \psi_2 \eta_{2,xx} \overline{i \psi_{2,t}}dxdt \\ = &s\frac{\gamma_2}{\sigma} \im  \iintq \psi_2 \eta_{2,xx} \overline{\psi_{2,t}} dxdt\\
				=& - s\frac{\gamma_2}{\sigma} \im\iintq  \psi_{2,x} \eta_{2,x} \overline{\psi_{2,t}}dxdt - s\frac{\gamma_2}{\sigma} \im \iintq  \psi_{2} \eta_{2,x} \overline{\psi_{2,tx}}dxdt \\
				&  + s\frac{\gamma_2}{\sigma}  \im \int_0^T \psi_2(t,1) \overline{\psi_{2,t}(t,1)} \eta_{2,x}(t,1)dt \\
				 = &s \frac{\gamma_2}{\sigma} \im \iintq \psi_{2,xt} \eta_{2,x} \overline{\psi_2}dxdt + s \frac{\gamma_2}{\sigma} \im \iintq \psi_{2,x} \eta_{2,xt} \overline{\psi_2}dxdt \\&- s\frac{\gamma_2}{\sigma}  \im \iintq  \psi_{2} \eta_{2,x} \overline{\psi_{2,xt}}dxdt + s\frac{\gamma_2}{\sigma}  \im \int_0^T \psi_2(t,1) \overline{\psi_{2,t}(t,1)} \eta_{2,x}(t,1)dt,
			\end{split}
		\end{equation}
		where we have applied integration by parts w.r.t. $x$ to the term involving $\eta_{2,xx}$, then w.r.t. $t$ to the term involving $\overline{\psi_{2,t}}$ in $Q_T$, and  using again the decay of $\eta_{2}$ at $t =0$ and $t = T$.

		Now, thanks to the fact that $-\im(z)=\im(\overline z)$ (for any $z\in \mathbb C$) in the third integral term in the last equality of \eqref{term-I-21-1}, we get 
		\begin{equation}\label{term-I-21}
			\begin{split}
				I_{21} =& 2s\frac{\gamma_2}{\sigma}  \im \iintq \psi_{2,xt} \eta_{2,x} \overline{\psi_2}dxdt + s\frac{\gamma_2}{\sigma} \im \iintq \psi_{2,x} \eta_{2,xt} \overline{\psi_2}dxdt \\
				& +  s\frac{\gamma_2}{\sigma} \im \int_0^T \psi_2(t,1) \overline{\psi_{2,t}(t,1)} \eta_{2,x}(t,1)dt\\ 
				 =& -J_2 - \frac{J_1}{2}  +   s\frac{\gamma_2}{\sigma}  \im \int_0^T \psi_2(t,1) \overline{\psi_{2,t}(t,1)} \eta_{2,x}(t,1)dt. 
			\end{split} 
		\end{equation}

Finally, the term $I_{32}$ is computed as follows. 
		\begin{equation}\label{term-I-32}
			\begin{split}
				I_{32} =& s\frac{\gamma_2}{\sigma}  \re \iintq i \psi_2 \eta_{2,t} \overline{\psi_{2,xx}}dxdt \\=& -s \frac{\gamma_2}{\sigma} \im \iintq \eta_{2,t} \psi_2 \overline{\psi_{2,xx}}dxdt  \\
				 = &s\frac{\gamma_2}{\sigma} \im \iintq \left(\eta_{2,tx} \psi_2 \overline{\psi_{2,x}}+ \eta_{2,t} \psi_{2,x} \overline{\psi_{2,x}} \right)dxdt \\&- s\frac{\gamma_2}{\sigma}  \im \int_0^T \eta_{2,t}(t,1) \psi_2(t,1) \overline{\psi_{2,x}(t,1)}dt\\
				 =& -s\frac{\gamma_2}{\sigma} \im \iintq \psi_{2,x} \eta_{2,xt} \overline{\psi_2}dxdt  \\&-  s\frac{\gamma_2}{\sigma} \im \int_0^T \eta_{2,t}(t,1) \psi_2(t,1) \overline{\psi_{2,x}(t,1)}dt \\
				 =& \frac{J_1}{2}  -   s\frac{\gamma_2}{\sigma}  \im \int_0^T \eta_{2,t}(t,1) \psi_2(t,1) \overline{\psi_{2,x}(t,1)}dt, 
			\end{split}
		\end{equation}
		where we have  used the fact that $$ \frac{s\gamma_2}{\sigma} \im \iintq \eta_{2,t}|\psi_{2,x}|^2dxdt = 0$$ since $\eta_{2,t}$ is real-valued function. Therefore, by adding  \eqref{term-I_11}, \eqref{term-I-21} and \eqref{term-I-32}, we get 
		\begin{equation}\label{Add-first-three}
			\begin{split}
				I_{11} + I_{21} + I_{32}  = & J_1 +   s\frac{\gamma_2}{\sigma} \im \int_0^T \psi_2(t,1) \overline{\psi_{2,t}(t,1)} \eta_{2,x}(t,1)dt \\&-  s\frac{\gamma_2}{\sigma} \im \int_0^T \eta_{2,t}(t,1) \psi_2(t,1) \overline{\psi_{2,x}(t,1)}dt,
			\end{split}
		\end{equation}
		where $J_1$ satisfies 
				\begin{equation}\label{bound-J-1}
			\begin{split}
				|J_1| =& \left| 2s \frac{\gamma_2}{\sigma} \im \iintq  \psi_{2,x} \eta_{2,xt} \overline{\psi_{2}} dxdt\right| \\
				 \leq& C\frac{\gamma_2}{\sigma}s\lambda T \iintq \xi^2_2 |\psi_2||\psi_{2,x}| dxdt\\
				 \leq& \frac{C}{\sigma^2} s^2 \lambda^2 T \iintq \xi^3_2 |\psi_{2}|^2 dxdt+ C  T \iintq \xi_2 |\psi_{2,x}|^2dxdt,
			\end{split}
		\end{equation}
		finishing step 1. 
		
		\vspace{0.2cm}
		
		\noindent $\bullet$ {\bf Step 2.}  Computations of the terms $I_{12}$ and $I_{22}$.
		
				\vspace{0.2cm}
				
		Let us, for $I_{12} $, to perform by integration by parts w.r.t. space variable to ensures that
		\begin{equation*}
			\begin{split}
				I_{12}  =  &2s\frac{\gamma_2^2}{\sigma^2}  \re \iintq \eta_{2,x}\psi_{2,x} \overline{\psi_{2,xx}}dxdt   \\
				 =& - 2s \frac{\gamma_2^2}{\sigma^2} \iintq \eta_{2,xx} |\psi_{2,x}|^2dxdt -   2s \frac{\gamma_2^2}{\sigma^2}\re \iintq \eta_{2,x} \psi_{2, xx} \overline{\psi_{2,x}}dxdt  \\
				&  + 2s\frac{\gamma_2^2}{\sigma^2}  \int_0^T \eta_{2,x} (t,1) |\psi_{2,x}(t,1)|^2dt - 2s\frac{\gamma_2^2}{\sigma^2}  \int_0^T \eta_{2,x} (t,0) |\psi_{2,x}(t,0)|^2dt.
			\end{split}
		\end{equation*}
		Recalling the expressions of $\eta_{2,x}$ and $\eta_{2,xx}$ from \eqref{derivatives-weights}, we get
		\begin{equation}\label{term-I-12}
			\begin{split}
				2 I_{12} =& 2s\lambda^2 c_2^2 \frac{\gamma_2^2}{\sigma^2} \iintq \xi_2 |\psi_{2,x}|^2 dxdt\\
				&  -  2s\lambda c_2 \frac{\gamma_2^2}{\sigma^2}   \int_0^T \xi_2 (t,1) |\psi_{2,x}(t,1)|^2dt + 2s\lambda c_2 \frac{\gamma_2^2}{\sigma^2}  \int_0^T \xi_2 (t,0) |\psi_{2,x}(t,0)|^2dt ,
			\end{split}
		\end{equation}
		where the boundary term $\psi_{2,x}(t,1)>0$ once we have $c_2<0$ (see \eqref{aux-func}). 
		
		\vspace{0.2cm}
		
		The term $I_{22}$ can be computed as 
		\begin{equation*}
			\begin{split}
				I_{22} = &s\frac{\gamma_2^2}{\sigma^2} \re \iintq \eta_{2,xx} \psi_2 \overline{\psi_{2,xx}}dxdt  \\
				 = &- s\frac{\gamma_2^2}{\sigma^2} \iintq \eta_{2,xx} |\psi_{2,x}|^2dxdt - s\frac{\gamma_2^2}{\sigma^2} \re \iintq \eta_{2,xxx} \psi_2 \overline{\psi_{2,x}}dxdt \\
				&+ s\frac{\gamma_2^2}{\sigma^2} \re \int_0^T \eta_{2,xx}(t,1) \psi_2(t,1) \overline{\psi_{2,x}(t,1)}dt .
			\end{split}
		\end{equation*}
		Again, using the expressions of $\eta_{2,xx}$ and $\eta_{2,xxx}$, thanks to \eqref{derivatives-weights}, we have that
		\begin{equation}\label{term-I-22}
			\begin{split}
				I_{22}=&s\lambda^2 c_2^2 \frac{\gamma_2^2}{\sigma^2} \iintq \xi_2 |\psi_{2,x}|^2 dxdt+ s\lambda^3 c_2^3 \frac{\gamma_2^2}{\sigma^2} \re \iintq \xi_2 \psi_2 \overline{\psi_{2,x}} dxdt\\
				& - s\lambda^2c_2^2\frac{\gamma_2^2}{\sigma^2} \re \int_0^T \xi_{2}(t,1) \psi_2(t,1) \overline{\psi_{2,x}(t,1)}dt.
			\end{split}
		\end{equation}

		Now, by adding the final expressions of $I_{12}$ and $I_{22}$, i.e,  \eqref{term-I-12} and \eqref{term-I-22}, yields that  
		\begin{equation}\label{add-I-12--I-22}
			\begin{split}
				I_{12}+ I_{22} =& 2s\lambda^2 c_2^2 \frac{\gamma_2^2}{\sigma^2} \iintq \xi_2 |\psi_{2,x}|^2dxdt  + s\lambda^3 c_2^3 \frac{\gamma_2^2}{\sigma^2} \re \iintq \xi_2 \psi_2 \overline{\psi_{2,x}}dxdt \\
				& -s\lambda c_2\frac{\gamma_2^2}{\sigma^2}   \int_0^T \xi_2 (t,1) |\psi_{2,x}(t,1)|^2dt + s\lambda c_2 \frac{\gamma_2^2}{\sigma^2}  \int_0^T \xi_2 (t,0) |\psi_{2,x}(t,0)|^2dt \\
				& - s\lambda^2c_2^2\frac{\gamma_2^2}{\sigma^2} \re \int_0^T \xi_{2}(t,1) \psi_2(t,1) \overline{\psi_{2,x}(t,1)}dt,
			\end{split}
		\end{equation}
		where the second term of the r.h.s. of \eqref{add-I-12--I-22}, denoted by $J_3$, satisfies
		\begin{equation}\label{bound-J-2}
			\begin{split}
				|J_3| =&  \left|s\lambda^3 c_2^3 \frac{\gamma_2^2}{\sigma^2} \re \iintq \xi_2 \psi_2 \overline{\psi_{2,x}}dxdt\right|\\
				 \leq &\frac{C}{\epsilon}  s \lambda^4\frac{\gamma_2^2}{\sigma^2} \iintq \xi_2 |\psi_2|^2dxdt+ \epsilon s \lambda^2\frac{\gamma_2^2}{\sigma^2} \iintq \xi_2 |\psi_{2,x}|^2dxdt \\
				 \leq & \frac{C}{\epsilon}   T^4 s \lambda^4 \frac{\gamma_2^2}{\sigma^2} \iintq \xi^3_2 |\psi_2|^2dxdt + \epsilon s \lambda^2 \frac{\gamma_2^2}{\sigma^2} \iintq \xi_2 |\psi_{2,x}|^2dxdt,
			\end{split}
		\end{equation}
		for any  $\epsilon>0$ small enough, where we have used the fact that $\xi_2 \leq CT^4 \xi_2^3$. 
		
			So, at this point, it is worth mentioning that the boundary integral consisting of $\psi_{2,x}(t,0)$ in \eqref{add-I-12--I-22} will lead to the observation term in our final Carleman estimate (since we have exerted a controlled force on the second component), finalizing the analysis in step 2.
		
\vspace{0.2cm}
		
		\noindent
		$\bullet$ {\bf Step 3.} Analysis of remaining terms $I_{13}$, $I_{23}$, $I_{31}$, and $I_{33}$. 
		
		\vspace{0.2cm}
		
		Let us first deal with the term $I_{13}$. Notice that the fact that $2\re(z)= z+\overline{z}$ (for any $z\in \mathbb C$), and then integrating by parts w.r.t. $x$, gives us
				\begin{equation*}
			\begin{split}
				I_{13}  =&  2 s^3 \frac{\gamma_2^2}{\sigma^2} \re \iintq \eta^3_{2,x} \psi_{2,x} \overline{\psi_2}dxdt  \\
				= &s^3 \frac{\gamma_2^2}{\sigma^2} \iintq \eta^3_{2,x} \left( \overline{\psi_2}\psi_{2,x} + \psi_2 \overline{\psi_{2,x}} \right) dxdt \\
				=&  s^3 \frac{\gamma_2^2}{\sigma^2} \iintq \eta^3_{2,x}  \left(\psi_2 \overline{\psi_{2}} \right)_x   dxdt \\
				=& - s^3 \frac{\gamma_2^2}{\sigma^2} \iintq 3\eta^2_{2,x} \eta_{2,xx} |\psi_2|^2 dxdt  + s^3\frac{\gamma_2^2}{\sigma^2} \int_0^T \eta^3_{2,x} |\psi_2(t,1)|^2dt \\
				 = &3s^3 \lambda^4 c_2^4 \frac{\gamma_2^2}{\sigma^2} \iintq \xi_2^3 |\psi_2|^2dxdt  -s^3 \lambda^3 c_2^3 \frac{\gamma_2^2}{\sigma^2}  \int_0^T  \xi_2^3(t,1) |\psi_2(t,1)|^2dxdt ,
			\end{split}
		\end{equation*}
		where the boundary term has a positive sign since $c_2<0$. 
		
			\vspace{0.2cm}
			
			Now, for the analysis of $I_{23}$ we just need to see that 
					\begin{equation*}
			\begin{split}
				I_{23} = &s^3 \frac{\gamma_2^2}{\sigma^2} \re \iintq \eta_{2,xx} |\eta_{2,x}|^2 |\psi_2|^2dxdt= - s^3\lambda^4 c_2^4\frac{\gamma_2^2}{\sigma^2} \iintq    \xi^3_2|\psi_2|^2 dxdt.
			\end{split}
		\end{equation*}
		
		Let us see the term $I_{31}$. For this integral, using again the decay properties of the weight function and integrating into the time variable, this term can be written as 
		\begin{equation*}
			\begin{split}
				I_{31} =&-  s \re \iintq \eta_{2,t} \psi_2 \overline{\psi_{2,t}}dxdt	\\=& -\frac{1}{2}  s \iintq \eta_{2,t} \left(\psi_2 \overline{\psi_{2,t}} + \overline{\psi_2} \psi_{2,t}  \right)dxdt  \\=& \frac{1}{2} s \iintq \eta_{2,tt} |\psi_2|^2dxdt   ,
			\end{split}
		\end{equation*}
		which satisfies 
$$
			|I_{31}| \leq Cs   T^2 \iintq  \xi^3_2 |\psi_2|^2 dxdt,
$$
		using the boundness of $\eta_{2,tt}$ from \eqref{derivatives-weights}.
		
	\vspace{0.2cm}
		
		For the last quantity $I_{33}$, since $\eta_{2,t}$ and $\eta_{2,x}$ are real valued functions, we get that  
$$
			I_{33} = s^3 \frac{\gamma_2}{\sigma} \re \left( i\iint \eta_{2,t}|\eta_{2,x}|^2  |\psi_2|^2 dxdt\right)  =0 .
$$
		
	\vspace{0.2cm}
		
		Adding each peach of $I_{13}$, $I_{23}$, $I_{31}$ and $I_{33}$, we have 
		\begin{equation}\label{Add-last-four}
			\begin{split}
				 I_{13} + I_{23} + I_{31}  + I_{33}  \geq & 2s^3 \lambda^4 c_2^4 \frac{\gamma_2^2}{\sigma^2} \iintq \xi_2^3 |\psi_2|^2dxdt
				- s^3 \lambda^3 c_2^3 \frac{\gamma_2^2}{\sigma^2} \int_0^T \xi^3_2(t,1) |\psi_2(t,1)|^2dt\\&  - C s T^2 \iintq \xi^3_2 |\psi_2|^2dxdt,
			\end{split}
		\end{equation} 
achieving the goal of step 3.

\vspace{0.2cm}
		
		\noindent
		$\bullet$ {\bf Step 4.} Finding an intermediate estimate. 
			\vspace{0.2cm}
			
			With the previous steps 1,2 and 3 in hand, adding all the terms $I_{kl}$ for $1\leq k, l\leq 3$, that is \eqref{Add-first-three}--\eqref{bound-J-1}, \eqref{add-I-12--I-22}--\eqref{bound-J-2}, and \eqref{Add-last-four}, we can get the following
		\begin{equation}\label{Combined-psi-2}
			\begin{split}
				 \sum_{1\leq k, l \leq 3} I_{kl} \geq& 2s^3 \lambda^4 c_2^4 \frac{\gamma_2^2}{\sigma^2} \iintq \xi_2^3 |\psi_2|^2 dxdt  + 
				2s\lambda^2 c_2^2 \frac{\gamma_2^2}{\sigma^2} \iintq \xi_2 |\psi_{2,x}|^2dxdt\\
				&- C s T^2 \iintq \xi^3_2 |\psi_2|^2 dxdt- \frac{C}{\sigma^2} s^2 \lambda^2 T \iintq \xi^3_2 |\psi_2|^2dxdt \\
				&- CT \iintq \xi_2 |\psi_{2,x}|^2 dxdt	- \frac{C}{\epsilon} T^4 s\lambda^4\frac{\gamma_2^2}{\sigma^2} \iintq \xi^3_2 |\psi_2|^2dxdt\\
				& - \frac{\gamma_2^2}{\sigma^2}\epsilon s\lambda^2 \iintq \xi_2 |\psi_{2,x}|^2 dxdt- s^3 \lambda^3 c_2^3 \frac{\gamma_2^2}{\sigma^2} \int_0^T \xi^3_2(t,1) |\psi_2(t,1)|^2dt\\
				&  - s\lambda c_2 \frac{\gamma_2^2}{\sigma^2} \int_0^T \xi_2(t,1) |\psi_{2,x}(t,1)|^2dt  + s\lambda c_2 \frac{\gamma_2^2}{\sigma^2} \int_0^T \xi_2(t,0) |\psi_{2,x}(t,0)|^2 dt  \\
				& - s\lambda^2c_2^2\frac{\gamma_2^2}{\sigma^2} \re \int_0^T \xi_{2}(t,1) \psi_2(t,1) \overline{\psi_{2,x}(t,1)}dt\\
				&  + s\frac{\gamma_2}{\sigma} \im \int_0^T \psi_2(t,1) \overline{\psi_{2,t}(t,1)} \eta_{2,x}(t,1) dt\\
				& -   s\frac{\gamma_2}{\sigma} \im \int_0^T \eta_{2,t}(t,1) \psi_2(t,1) \overline{\psi_{2,x}(t,1)}dt  .
			\end{split}
		\end{equation}

		Following a similar approach as the case of $\psi_2$, we can compute the same for $\psi_1$, for this, let us denote
$$
			\re  \iintq M_1 \psi_1 \overline{M_2 \psi_1}dxdt = \sum_{1\leq k,l\leq 3} H_{kl} .
$$
In this case, one can find 
		\begin{equation}\label{combined-psi-1}
			\begin{split}
				 \sum_{1\leq k, l \leq 3} H_{kl}\geq &2s^3 \lambda^4 c_1^4 \gamma_1^2 \iintq \xi_1^3 |\psi_1|^2dxdt  + 
				2s\lambda^2 c_1^2 \gamma_1^2 \iintq \xi_1 |\psi_{1,x}|^2dxdt\\ 
				&- C s T^2 \iintq \xi^3_1 |\psi_1|^2 dxdt - C  s^2 \lambda^2 T \iintq \xi^3_1 |\psi_1|^2 dxdt\\
				&  - CT \iintq \xi_1 |\psi_{1,x}|^2dxdt- \frac{C}{\epsilon} T^4 s\lambda^4 \gamma_1^2\iintq \xi^3_1 |\psi_1|^2dxdt\\
				&  - \epsilon s\lambda^2 \iintq \xi_1 |\psi_{1,x}|^2dxdt- s^3 \lambda^3 c_1^3 \gamma_1^2 \int_0^T \xi^3_1(t,1) |\psi_1(t,1)|^2dt \\
				&- s\lambda c_1 \gamma_1^2 \int_0^T \xi_1(t,1) |\psi_{1,x}(t,1)|^2 dt +  s\lambda c_1 \gamma_1^2 \int_0^T \xi_1(t,0) |\psi_{1,x}(t,0)|^2 dt  \\
				& - s\lambda^2c_1^2\gamma_1^2 \re \int_0^T \xi_{1}(t,1) \psi_1(t,1) \overline{\psi_{1,x}(t,1)} dt \\& + s\gamma_1  \im \int_0^T \psi_1(t,1) \overline{\psi_{1,t}(t,1)} \eta_{1,x}(t,1)dt  \\
				& -   s\gamma_1  \im \int_0^T \eta_{1,t}(t,1) \psi_1(t,1) \overline{\psi_{1,x}(t,1)}dt .
			\end{split}
		\end{equation}
		
\vspace{0.2cm}

We are now in a position to combine the above estimate. For simplicity, we denote $\nu_1=\gamma_1$ and $ \nu_2=\frac{\gamma_2}{\sigma}.$
		Choose $\epsilon>0$ small enough in \eqref{Combined-psi-2} and \eqref{combined-psi-1} so that, combining those estimates, we have 
		\begin{equation} 
			\begin{split}\label{carleman lemma}
				&\frac{1}{2}\sum_{j=1}^{2}\iintq \left(|M_1 \psi_j|^2 + |M_2 \psi_j|^2\right)dxdt\\
				&+\sum_{j=1}^{2} 2s^3\lambda^4 c_j^4 \nu_j^2 \iintq  \, \xi_j^3|\psi_j|^2dxdt +\sum_{j=1}^{2} 2s \lambda^2 \nu_j  \iintq \xi_j |\psi_{j,x}|^2dxdt \\
				& - C \left(\frac{1}{\sigma^2}s^2 \lambda^2 T + T^4 s\lambda^4 \frac{\gamma^2}{\sigma^2} \right)   \iint_{Q_T} \xi^3_2 |\psi_2|^2dxdt  - C T  \iint_{Q_T} \xi^3_2 |\psi_{2,x}|^2dxdt  \\
				& - C \left(s^2 \lambda^2 T + T^4 s\lambda^4 \gamma_1^2 \right)   \iint_{Q_T} \xi^3_1|\psi_1|^2dxdt  - C T  \iint_{Q_T} \xi^3_1 |\psi_{1,x}|^2 dxdt \\ 
				&- s^3 \lambda^3 \sum_{j=1}^{2}c_j^3 \nu_j^2 \int_0^T \xi^3_j(t,1) |\psi_j(t,1)|^2dt - s \lambda \sum_{j=1}^{2}c_j \nu_j^2 \int_0^T \xi_j(t,1) |\psi_{j,x}(t,1)|^2 dt  \\
				& +s \lambda \sum_{j=1}^{2}c_j \nu_j^2 \int_0^T \xi_j(t,0) |\psi_{j,x}(t,0)|^2dt  - s\lambda^2\sum_{j=1}^{2} c_j^2\nu_j^2 \re \int_0^T \xi_{j}(t,1) \psi_j(t,1) \overline{\psi_{j,x}(t,1)}dt   \\&+ s\sum_{j=1}^{2} \nu_j  \im \int_0^T \psi_j(t,1) \overline{\psi_{j,t}(t,1)} \eta_{j,x}(t,1) dt  \\
				& -   s\sum_{j=1}^{2} \nu_j  \im \int_0^T \eta_{j,t}(t,1) \psi_j(t,1) \overline{\psi_{j,x}(t,1)}dt \leq C\sum_{j=1}^{2} \iintq |F_j|^2dxdt. 
			\end{split}
		\end{equation} 
		
		In the above, it is clear that there exist some positive constants  $\lambda_0$, $\mu_0$ such that if we choose $\lambda\geq \lambda_0$ and $s\geq s_0:=\mu_0(T+T^2)$, then the lower order integrals 
		\begin{align*}
			&C \left(\frac{1}{\sigma^2}s^2 \lambda^2 T + T^4 s\lambda^4 \frac{\gamma^2}{\sigma^2} \right)   \iint_{Q_T} \xi^3_2 |\psi_2|^2dxdt , \\ 
			&C \left(s^2 \lambda^2 T + T^4 s\lambda^4 \gamma_1^2 \right)   \iint_{Q_T} \xi^3_1|\psi_1|^2dxdt , \\
			&C T  \iint_{Q_T} \xi^3_2 |\psi_{2,x}|^2dxdt , \\ 
					\end{align*}
					and
			\begin{align*}
			C T  \iint_{Q_T} \xi^3_1 |\psi_{1,x}|^2dxdt,
					\end{align*}
		can be absorbed by the corresponding leading integrals 
		\begin{align*}
			\sum_{j=1}^{2} 2s^3\lambda^4 c_j^4 \nu_j^2 \iintq  \, \xi_j^3|\psi_j|^2dxdt \ \text{ and }\  \sum_{j=1}^{2} 2s \lambda^2 \nu_j  \iintq \xi_j |\psi_{j,x}|^2dxdt .
		\end{align*}
		Now, it remains to find the proper estimates for the boundary integrals. We do this in the following step. 
		
			\vspace{0.2cm} 
			
		\noindent
		$\bullet$ {\bf Step 5.} Computing the boundary integrals.
			\vspace{0.2cm}    
		
		We split this step into several parts. We hereby recall the fact that 
		$$\xi_1(t,1)=\xi_2(t,1), \ \ \eta_1(t,1)=\eta_2(t,1), \ \ \psi_1(t,1)=\psi_{2}(t,1),$$
		and thus, from now onwards, we shall replace all the quantities $\xi_1(t,1)$, $\eta_1(t,1)$ and $\psi_1(t,1)$  by $\xi_2(t,1)$, $\eta_2(t,1)$ and  $\psi_2(t,1)$  respectively. Moreover, we will denote the six boundary terms in \eqref{carleman lemma}  in the order presented above as $B_k$ given by $$B_k=\sum_{j=1,2}B_{k,j}, \quad 1\leq k\leq 6.$$ Let us now give the details of each part. 
		
\vspace{0.2cm}
		\noindent -- {\bf Part A:} Computing $B_1:=B_{11} + B_{12}$.
\vspace{0.2cm}

First notice that, the term
		$$B_{12}=- s^3 \lambda^3 c^3_2 \frac{\gamma_2^2}{\sigma^2} \int_0^T \xi^3_2(t,1) |\psi_{2}(t,1)|^2dt$$ 
		is positive, since $c_2<0.$
		We  need to absorb the negative term  
		$$B_{11}=- s^3 \lambda^3 c^3_1 \gamma_1^2 \int_0^T \xi^3_2(t,1) |\psi_{2}(t,1)|^2dt$$ 
		by the term $B_{12}.$  To do that, since we have $c_1=1$ and, by \eqref{Value-sigma} that 
	$$\sigma = -\frac{c_2\gamma_2}{\gamma_1}$$  we have, by  adding $B_{11}$ and $B_{12}$, that 
			\begin{equation}  
			\begin{aligned}\label{first-leading-2}
				B_1&=B_{11} + B_{12} \\&= s^3\lambda^3 \left(|c_2|^3 \frac{\gamma^2_2}{\sigma^2} - \gamma_1^2  \right) \int_0^T \xi^3_2(t,1) |\psi_{2}(t,1)|^2dt  \\
				& = s^3\lambda^3 \gamma_1^2 \left(|c_2| -1 \right) \int_0^T \xi^3_2(t,1) |\psi_{2}(t,1)|^2dt,
			\end{aligned} 
		\end{equation} 
		which has a positive sign since by choice we have $|c_2|>3$.

		\vspace{0.2cm}
		\noindent -- {\bf Part B:} $B_2:=B_{21} + B_{22}$ estimate.
\vspace{0.2cm}
		
		Now, on one hand, the integral
		$$ B_{22} =- s \lambda c_2 \frac{\gamma_2^2}{\sigma^2} \int_0^T \xi_2(t,1) |\psi_{2,x}(t,1)|^2dt $$
		is positive since $c_2<0$ and we shall use this term to absorb the negative 
		integral
		$$B_{21}=- s \lambda c_1 \gamma_1^2 \int_0^T \xi_2(t,1) |\psi_{1,x}(t,1)|^2 dt. $$ 
	To see that, thanks to  the boundary conditions \eqref{boundary-psi}, we can ensure that   
		\begin{equation*}
			\gamma_1\psi_{1,x}(t,1)=-\alpha\psi_2(t,1)-\frac{\gamma_2}{\sigma}\psi_{2,x}(t,1)+ s\lambda \xi_2(t,1) \psi_2(t,1) \left(\gamma_1 c_1 + \frac{\gamma_2}{\sigma} c_2 \right).
		\end{equation*}
		Then, by using Young's inequality, we  obtain
		\begin{equation} 
			\begin{split}\label{i21}
				B_{21}  =& s \lambda c_1  \int_0^T \xi_2(t,1) |\gamma_1 \psi_{1,x}(t,1)|^2dt \\ 
				 \leq &3s \lambda c_1 \int_{0}^{T} \xi_2(t,1)  \left|\frac{\gamma_2}{\sigma} \psi_{2,x}(t,1)\right|^2dt  + 3s \lambda c_1 \alpha^2\int_{0}^{T} \xi_2(t,1)  |\psi_2(t,1)|^2dt  \\ 
				&  +3s^3 \lambda^3 c_1 \left(\gamma_1 c_1 + \frac{\gamma_2}{\sigma} c_2 \right)^2\int_{0}^{T} \xi^3_2(t,1) |\psi_2(t,1)|^2 dt \\
				\leq &3s \lambda c_1 \int_{0}^{T} \xi_2(t,1)  \left|\frac{\gamma_2}{\sigma} \psi_{2,x}(t,1)\right|^2dt + CT^4 s \lambda c_1 \alpha^2\int_{0}^{T} \xi^3_2(t,1)  |\psi_2(t,1)|^2dt  \\
				&
				+3s^3 \lambda^3 c_1 \left(\gamma_1 c_1 + \frac{\gamma_2}{\sigma} c_2 \right)^2\int_{0}^{T} \xi^3_2(t,1) |\psi_2(t,1)|^2 dt . 
			\end{split}
		\end{equation} 
		So, adding $B_{21}$ and $B_{22}$,  the fact that $c_1=1$, and the choice of $\sigma$ given by \eqref{Value-sigma}, yields that
		\begin{equation}\label{leading-second}
			\begin{split}
				B_2=&B_{21} + B_{22} \\
				\geq &\frac{s\lambda \gamma_1^2}{c_2^2} \left(|c_2| - 3   \right) \int_0^T \xi_2(t,1) \left| \psi_{2,x}(t,1)\right|^2dt -  CT^4 s \lambda \alpha^2\int_{0}^{T} \xi^3_2(t,1)  |\psi_2(t,1)|^2dt.  
			\end{split}
		\end{equation}
		Note that, here we exactly need the assumption  $|c_2|>3$ (in other words, $\kappa>3$ in \eqref{Value-sigma} since $c_2=-\kappa$) to make the first integral of \eqref{leading-second} positive. 

			\vspace{0.2cm}
		\noindent -- {\bf Part C:} Computing each term of $B_3:=B_{31}+ B_{32}$, that is, $B_{31}$ and $B_{32}$.
\vspace{0.2cm}
	
			Let us look at these terms separately. We have 
		\begin{align}\label{term-I-3}
			B_{31}=s \lambda c_1 \gamma_1^2 \int_0^T \xi_1(t,0) |\psi_{1,x}(t,0)|^2dt\geq 0, 
		\end{align} 
		since $c_1=1$ and $\xi_1(t,0)>0.$ Now,  using the expression \eqref{variables-psi der},
		we find 
		\begin{equation*}
			\psi_{2,x}(t,x)=e^{-s\eta_2(t,x)}\vphi_{2,x}(t,x)-se^{-s\eta_{2}(t,x)}\eta_{2,x}(t,x)\vphi_{2}(t,x).
		\end{equation*}
		But $\vphi_2(t,0)=0$ and so,
		\begin{equation*}
			\psi_{2,x}(t,0)=e^{-s\eta_2(t,0)}\vphi_{2,x}(t,0).
		\end{equation*}
		Therefore,  we have 
		\begin{equation*}
			\begin{split} 
				|B_{32}|& =s \lambda \left| c_2 \frac{\gamma_2^2}{\sigma^2} \int_0^T \xi_2(t,0)e^{-2s\eta_2(t,0)} |\vphi_{2,x}(t,0)|^2 dt \right| \\
				&\leq \frac{s \lambda \gamma_1^2}{|c_2|}\int_0^T \xi_2(t,0)e^{-2s\eta_2(t,0)} |\vphi_{2,x}(t,0)|^2 dt,
			\end{split}
		\end{equation*}
		which is indeed the observation integral for our final Carleman estimate.

			\vspace{0.2cm}
		\noindent -- {\bf Part D:} Analysing the terms $B_{41}$ and $B_{42}$.
\vspace{0.2cm}
	
		Next, we will compute the fourth boundary term as follows. First, observe that  
				\begin{align*}
			|B_{41}|=&\left| s\lambda^2 c_1^2\gamma_1^2 \re \int_0^T \xi_{2}(t,1) \psi_2(t,1) \overline{\psi_{1,x}(t,1)} dt \right| \\
			\leq& \epsilon s \lambda  \int_{0}^{T} \xi_{2}(t,1) |\gamma_1 {\psi_{1,x}(t,1)}|^2 dt+\frac{C}{\epsilon}s \lambda^3 T^4 \int_{0}^{T}\xi_{2}^3(t,1) |{\psi_{2}(t,1)}|^2 dt,
		\end{align*}
	where $C=C(\gamma_1)>0$. Here we have used Young's inequality, the estimate $\xi_2 \leq CT^4 \xi_2^3$, and the fact that $c_1=1$. In the above estimate, we apply   \eqref{i21}, so that one has 
		\begin{equation} 
			\begin{split}\label{bound-I-41}
				|B_{41}|\leq&  3\epsilon s \lambda \int_{0}^{T} \xi_{2}(t,1) \left|\frac{\gamma_2}{\sigma}{\psi_{2,x}(t,1)}\right|^2 dt\\
				&+ \left(C\epsilon T^4 \alpha^2 s \lambda + \frac{C}{\epsilon} s\lambda^3T^4\right) \int_{0}^{T} \xi^3_2(t,1)  |\psi_2(t,1)|^2dt\\
				 = & \frac{3\epsilon s\lambda \gamma_1^2}{c_2^2} \int_{0}^{T} \xi_{2}(t,1) \left|\psi_{2,x}(t,1)\right|^2 dt
				\\&+ \left(C\epsilon T^4 \alpha^2 s \lambda + \frac{C}{\epsilon} s\lambda^3T^4\right) \int_{0}^{T} \xi^3_2(t,1)  |\psi_2(t,1)|^2dt,
			\end{split}
		\end{equation}
		where we have used the choice of $\sigma$ as given by \eqref{Value-sigma} and $c_2$ as \eqref{aux-func}.
		
		Similarly, one can estimate the term $B_{42}$ as follows:
		\begin{equation}\label{bound-I-42}
		\begin{split}
			|B_{42}|
			\leq & \epsilon s \lambda |c_2| \frac{\gamma_2^2}{\sigma^2}\int_{0}^{T} \xi_{2}(t,1) |{\psi_{2,x}(t,1)}|^2 dt + \frac{C}{\epsilon} s \lambda^3 T^4 \int_{0}^{T}\xi_{2}^3(t,1) |{\psi_{2}(t,1)}|^2 dt\\
			= & \frac{\epsilon s \lambda \gamma_1^2}{|c_2|}\int_{0}^{T} \xi_{2}(t,1) |{\psi_{2,x}(t,1)}|^2 dt + \frac{C}{\epsilon} s \lambda^3 T^4 \int_{0}^{T}\xi_{2}^3(t,1) |{\psi_{2}(t,1)}|^2 dt,
		\end{split}
		\end{equation}
		finalizing this part.

			\vspace{0.2cm}
		\noindent -- {\bf Part E:} Analysis of $B_5:=B_{51}+B_{52}$.
\vspace{0.2cm}

		Let us look into the terms of $B_5.$ To do that, these terms can be viewed as
		\begin{equation}\label{term-I-5}
		\begin{split}
			B_{5}=&B_{51}+B_{52}\\=& s\sum_{j=1}^{2} \nu_j  \im \int_0^T \psi_j(t,1) \overline{\psi_{j,t}(t,1)} \eta_{j,x}(t,1)dt\\ 
			=&-s\lambda \left(c_1\gamma_1+c_2 \frac{\gamma_2}{\sigma}\right) \im \int_0^T \psi_1(t,1) \overline{\psi_{1,t}(t,1)} \xi_{1}(t,1)dt,
			\end{split}
		\end{equation}
where we have used that $\eta_{j,x} = - \lambda \xi_j c_j .$

\vspace{0.2cm}

It is important to point out that this term is difficult to absorb in terms of the leading terms because of the appearance time derivative term $\overline{\psi_{1,t}(t,1)}$, and this is the main reason why we have chosen the parameter $\sigma$ as in \eqref{Set-sigma} (in other words, \eqref{Value-sigma}) and $c_2=-\kappa$ in \eqref{aux-func}. Thanks to those choices, one readily has
		$$\left(c_1\gamma_1+\frac{c_2 \gamma_2}{\sigma}\right)=0,$$ once $c_1=1$, and this makes the quantity $B_5$ equal zero.
		
		\vspace{0.2cm}
		\noindent -- {\bf Part F:} Estimates to $B_{61}$ and $B_{62}$.
\vspace{0.2cm}

		Finally, we compute the terms $B_{6}$.  Using that  $$|\eta_{j,t}| \leq CT \xi^2_j\quad \text{and}\quad \sigma=\frac{\kappa\gamma_2}{\gamma_1}= -\frac{c_2\gamma_2}{\gamma_1},$$ we get 
		\begin{equation}\label{bound-I-62}
			\begin{aligned}
				|B_{62}| & \leq  C T s \frac{\gamma_2}{\sigma} \int_0^T |\psi_2(t,1)| |{\psi_{2,x}(t,1)}| |\xi_{2}(t,1)|^2 dt \\
				& \leq \frac{\gamma^2_1}{c^2_2}  \int_{0}^{T} \xi_{2}(t,1) \left|{\psi_{2,x}(t,1)}\right|^2 dt 
				+Cs^2T^2 \int_{0}^{T}\xi_{2}^3(t,1) |{\psi_{2}(t,1)}|^2 dt .
			\end{aligned}
		\end{equation}
		Using the expression of $\gamma_1\psi_{1,x}(t,1)$ from \eqref{boundary-psi}, we further get 
		\begin{equation}\label{bound-I-61}
			\begin{aligned}
				|B_{61}|
				\leq & CTs \gamma_1 \int_0^T |\psi_2(t,1)| |{\psi_{1,x}(t,1)}| \xi^2_{2}(t,1) dt \\
				\leq &  \int_{0}^{T} \xi_{2}(t,1) |\gamma_1 {\psi_{1,x}(t,1)}|^2 dt + Cs^2  T^2 \int_{0}^{T}\xi_{2}^3(t,1) |{\psi_{2}(t,1)}|^2 dt \\
				\leq & \frac{3\gamma_1^2}{c_2^2}   \int_{0}^{T} \xi_{2}(t,1) \left|{\psi_{2,x}(t,1)}\right|^2 dt + C\left(\alpha^2T^4  +  s^2  T^2\right) \int_{0}^{T}\xi_{2}^3(t,1) |{\psi_{2}(t,1)}|^2 dt.
			\end{aligned}
		\end{equation}
	Here, thanks to the quantity \eqref{i21} we can find the estimate for the following term $$\int_0^T \xi_2(t,1)|\gamma_1 \psi_{1,x}(t,1)|^2 dt.$$

			\vspace{0.2cm} 
			
		\noindent
		$\bullet$ {\bf Step 6.} Getting rid of the lower-order boundary integrals.
					\vspace{0.2cm}    

Finally, thanks to the computation of the boundary terms (step 5), we can add all the parts of this step to have the estimate for $B_{k}$ for $1\leq k\leq 6$. Putting together \eqref{first-leading-2}, \eqref{leading-second}, \eqref{term-I-3}, \eqref{bound-I-41}, \eqref{bound-I-42}, \eqref{term-I-5}, \eqref{bound-I-62} and \eqref{bound-I-61}, we find that
		\medskip 
		\begin{equation}\label{Add-boundary-terms}
			\begin{split}
				\sum_{k=1}^{6} B_{k} \geq  & s^3\lambda^3 \gamma_1^2 \left(|c_2| -1 \right) \int_0^T \xi^3_2(t,1) |\psi_{2}(t,1)|^2dt \\&+ \frac{s\lambda \gamma^2_1}{c^2_2} \left(|c_2| - 3   \right) \int_0^T \xi_2(t,1) \left|\psi_{2,x}(t,1)\right|^2dt\\
				& -  CT^4 s \lambda  \alpha^2\int_{0}^{T} \xi^3_2(t,1)  |\psi_2(t,1)|^2dt \\& - \frac{\epsilon s\lambda \gamma_1^2}{c_2^2} \left( 3 + |c_2|  \right) \int_0^T \xi_2(t,1) \left| \psi_{2,x}(t,1)\right|^2dt \\
				& - \left(C\epsilon T^4 \alpha^2 s\lambda + \frac{2C}{\epsilon} s\lambda^3 T^4  \right) \int_0^T \xi^3_2(t,1) |\psi_{2}(t,1)|^2dt\\&
				- \frac{4\gamma_1^2}{c_2^2} \int_0^T \xi_2(t,1) \left|\psi_{2,x}(t,1)\right|^2dt \\
				& - C\left(\alpha^2T^4  +   s^2  T^2\right) \int_{0}^{T}\xi_{2}^3(t,1) |{\psi_{2}(t,1)}|^2 dt \\&- \frac{s \lambda \gamma_1^2}{|c_2|} s\lambda \int_0^T \xi_2(t,0) e^{-2s\eta_2(t,0)} |\vphi_{2,x}(t,0)|^2 dt  .
			\end{split}
		\end{equation}
		In the above we fix $\epsilon>0$ small enough and choose  $\lambda \geq \lambda_0$ and $s\geq s_0=\mu_0(T+T^2)$ (for  $\lambda_0>0$, $\mu_0>0$ large enough) so that all the lower order boundary integrals can be absorbed by the first two leading integrals of \eqref{Add-boundary-terms}. 
		
		As a consequence,  from   \eqref{carleman lemma}, one has 
		\begin{equation*}
			\begin{split}
				&s^3\lambda^4 \iintq \left(\xi^3_1 |\psi_1|^2 + \xi^3_2 |\psi_2|^2 \right)dxdt + s\lambda^2 \iintq \left(\xi_1 |\psi_{1,x}|^2 + \xi^3_2 |\psi_{2,x}|^2 \right) dxdt\\   
				&+	s^3\lambda^3 \int_0^T \xi^3_2(t,1) |\psi_{2}(t,1)|^2dxdt + s\lambda  \int_0^T \xi_2(t,1) \left|\psi_{2,x}(t,1)\right|^2dt \\
				\leq& C \iintq (|F_1|^2 + |F_2|^2)dxdt  + C s\lambda \int_0^T e^{-2s\eta_2(t,0)} \xi_2(t,0) |\vphi_{2,x}(t,0)|^2 dt,
			\end{split}
		\end{equation*}
		for all $\lambda \geq \lambda_0$ and $s\geq s_0$. Therefore, the previous inequality gives us the required Carleman estimate  \eqref{finalcarlemanestimate}, and the proof of Theorem \ref{Thm.1} is complete.
	\end{proof}

	\subsection{Observability inequality and its application}
	This section is devoted to proving a suitable observability inequality as a consequence of the Carleman estimate \eqref{finalcarlemanestimate}, which is the key point to deduce the null-controllability of the system \eqref{linear}. 
	
	\begin{proposition}
		For any $\zeta:= (\zeta_1 , \zeta_2 ) \in \mathcal{H}$, the associated solution of the adjoint system \eqref{adjoint} $$\vphi:= ( \vphi_1 , \vphi_2 )\in C([0,T]; \mathcal{H})$$ satisfies the following boundary observation
		\begin{equation}\label{Obs}
			\|\vphi(0)\|^2_{\mathcal H}  \leq C \int_0^T |{\partial_x \vphi_2}(t,0)|^2 dt. 
		\end{equation}  
	\end{proposition} 
	
	\begin{proof}
Consider $\zeta\in \mathcal{H}$. Thanks to the Carleman inequality, given by Theorem \ref{Thm.1} to the solution $(\vphi_1,\vphi_2)$ of \eqref{adjoint} with $$L_1\vphi_1= \alpha_1 \vphi_1 \quad \text{and}\quad L_2\vphi_2 =\frac{\alpha_2\vphi_2}{\sigma}\quad \left(\sigma=\frac{\kappa \gamma_2}{\gamma_1}\right),$$ 
gives us that
		\begin{equation}\label{Observation_1} 
			\begin{split}
				&s^3\lambda^4 \iintq \left(e^{-2s\eta_1}\xi^3_1 |\vphi_1|^2 + e^{-2s\eta_2}\xi^3_2 |\vphi_2|^2 \right)dxdt \\&+ s\lambda^2 \iintq \left(e^{-2s\eta_1}\xi_1 |\vphi_{1,x}|^2 + e^{-2s\eta_2}\xi_2 |\vphi_{2,x}|^2 \right) dxdt  \\  
				\leq &C \iintq (e^{-2s\eta_1}|\vphi_1|^2 + e^{-2s\eta_2}|\vphi_2|^2) dxdt+ C s\lambda \int_0^T e^{-2s\eta_2(t,0)} \xi_2(t,0) |\vphi_{2,x}(t,0)|^2dt.
			\end{split}
		\end{equation} 
		
		Now, using the fact that $1\leq 8T^6\xi^3_j$ for $j=1,2$, we can easily absorb the first two integrals of the r.h.s. of \eqref{Observation_1} by the associated leading integrals for any $s\geq CT^2$. Thus, we have the following inequality from \eqref{Observation_1}, 
		\begin{equation}\label{Observation_2}
			\begin{split} 
				s\lambda^2 \iint_{Q_T}&\left(   
				e^{-2s\eta_1} \xi_1 |\vphi_{1,x}|^2 +
				e^{-2s\eta_2} \xi_2 |\vphi_{2,x}|^2 \right)  dxdt\\
				&				\leq C s\lambda \int_0^T \xi_2(t,0) e^{-2s\eta_2(t,0)} |\vphi_{2,x}(t,0)|^2dt.
			\end{split}
		\end{equation}
		Moreover, for some $m, M>0$, we have the following relations 
		\begin{align*}
			e^{-2s\eta_j} \xi_j \geq m \ \text{ in } (T/4, 3T/4)\times (0,1),    \quad \text{and} \quad  	e^{-2s\eta_j} \xi_j \leq M \  \text{ in } Q_T, \ \ \text{for $j=1,2$} ,
		\end{align*}
which together with \eqref{Observation_2}, yields that
$$
			\int_{T/4}^{3T/4} \int_0^1 \left(|\vphi_{1,x}|^2 + |\vphi_{2,x}|^2\right)dxdt
			 \leq C \int_0^T |\vphi_{2,x}(t,0)|^2dt. 
$$
		Then, from the estimate given by \eqref{esti-adj}, one can deduce that
$$
			\| \vphi(0) \|^2_{\mc H} \leq C_T 	\int_{T/4}^{3T/4} \int_0^1 \left(|\vphi_{1,x}|^2 + |\vphi_{2,x}|^2\right) dxdt
			 \leq C \int_0^T |\vphi_{2,x}(t,0)|^2dt,
$$
		and hence the required observability inequality \eqref{Obs} follows. 
	\end{proof} 
We are now in a position to prove the first main result.
	\begin{proof}[Proof of Theorem \ref{th-main}] The proof follows the classical Hilbert uniqueness method introduced by Lions \cite{Lions}. Once we have the above observability inequality, then one can prove the existence of a boundary control $h\in L^2(0,T)$ such that $(u,v)$ solutions of \eqref{linear} with boundary conditions \eqref{boundary-1}-\eqref{control-v}
		 satisfies \eqref{null}.
		\end{proof}


\section{Rapid exponential stabilization}\label{sec3}

This section is devoted to studying the boundary stabilization issues for \eqref{linear}-\eqref{boundary-1} with a single boundary control force acting on the component $v$, namely \eqref{control-v}. More precisely, we construct a stationary feedback law $h(t)$, of the form ${F}_{\omega}(u(t,\cdot), v(t,\cdot))$ such that the solution of the closed-loop system decays exponentially to zero at any prescribed decay rate.  

The approach employed in this section was first introduced by Komornik in \cite{Kom-1997} and has also been studied by Vaste \cite{Vest} and Urquiza \cite{Urquiza}, which one will be applied in our context and is the key argument of this section.

\subsection{Gramian method}
Let us consider the abstract control system
\begin{equation} \label{eq:abs}
	\begin{cases}
		\dot{y}(t)=\mc Ay(t)+ \mc B h(t) , \quad t\in (0,T),\\
		y(0)=y_0,
	\end{cases}
\end{equation}
where $y(t)\in \mc H, y_0\in \mc H, h\in L^2(0,T)$, $\mc B$ is an unbounded operator from $\cplx$ to $\mc H$. $\mc A:\mathcal{D(A)}\subset \mc H\to \mc H$ is an unbounded operator and $\mc D(\mc A)$ is dense in $\mc H.$ To employ the method of Urquiza, one needs to take the following assumptions on the operator $\mc A$ and $\mc B$
\begin{itemize}
	\item[(H1)] The skew-adjoint operator $\mc A$ is an infinitesimal generator of a strongly continuous group $\{e^{t \mc A}\}_{t\in \rea}$ on $\mc H$.
	\item[(H2)] The operator $\mc B:\cplx\to \mc D(\mc A^*)'$ is linear and continuous.
	\item[(H3)] \textit{(Regularity property)} For every $T>0$ there exists $C(T)>0$ such that
	\begin{equation*}
		\int_{0}^{T}\abs{\mc B^*e^{-t\mc A^*} y}^2dt\leq C \norm{y}^2_{\mc H}, \quad \forall \,y \in \mc D(\mc A^*).
	\end{equation*}
	\item[(H4)] \textit{(Controllability property)} There are two constants $T>0$ and $c(T)>0$ such that
	\begin{equation*}
		\int_{0}^{T}\abs{\mc B^*e^{-t\mc A^*} y}^2dt\geq c \norm{y}^2_{\mc H}, \quad \forall \,  y \in \mc D(\mc A^*).	
	\end{equation*}
\end{itemize}

\begin{remark}\label{irena}
It is important to mention that the hypothesis (H3) is known as direct inequality (see, for instance, \cite{Lions}) or the admissibility of the control operator (see, e.g., \cite{Tucsnak}). This property ensures the well-posedness of the control system \eqref{linearsource} as we can see, for example, in \cite[Proposition A.1.]{Irena} for hyperbolic and Petrowski PDEs. 
\end{remark}

With these hypotheses in hand, the next result holds (for details, see \cite[Theorem 2.1]{Urquiza}). Its proof mainly relies on general results about the algebraic Riccati equation associated with the linear quadratic regulator problem (see \cite{FlaLaTri}).
\begin{theorem}\label{thm}
	Consider operators $\mc A$ and $\mc B$ under assumptions (H1)-(H4). For any $\omega>0$, we have
	\begin{itemize}
		\item[(i)] The symmetric positive operator $\Lambda_{\omega}$ defined by
				\begin{equation*}
			\ip{\Lambda_{\omega}x}{z}_{\mc H}=\int_{0}^{\infty}\ip{\mc B^*e^{-\tau(\mc A+\omega I)^*}x}{\mc B^*e^{-\tau(\mc A+\omega I)^*}z}_{\cplx}d\tau, \, \forall \, x, z\in \mc H
		\end{equation*}
		is coercive and is an isomorphism on $\mc H$.
		\item[(ii)] Let $F_{\omega}:=-\mc B^*\Lambda_{\omega}^{-1}$. The operator $\mc A+\mc BF_{\omega}$ with $\mc D(\mc A+\mc BF_{\omega})=\Lambda_{\omega}(\mc D(\mc A^*))$ is the infinitesimal generator of a strongly continuous semigroup on $\mc H$. 		\item[(iii)] The closed-loop system \eqref{eq:abs} with $h=F_{\omega}(y)$ is exponentially stable, that is,
		\begin{equation*}
			\norm{e^{t(\mc A+\mc B F_{\omega})}y}_{\mc H}\leq C e^{\left(-2\omega+g(-\mc A)\right) t} \norm{y}_{\mc H}, \ \forall y \in \mc H,\end{equation*} where $C$ is a positive constant. 
				\end{itemize}
\end{theorem}

We will use Theorem \ref{thm} to prove the exponential stabilization of the coupled Schrodinger equation \eqref{linear}-\eqref{boundary-1}-\eqref{control-v}  with boundary feedback law. To apply it, we need to verify all the assumptions (H1)--(H4) hold for our system \eqref{linear}. Let us do it in the next subsection.  

\subsection{Verification of the hypotheses} It is well-known that a fundamental solution of the Schr\"odinger system can be obtained by the Fourier expansion, see for instance \cite{CoGaMo,JaKo}. So, in this way, considering the eigenvalues and the eigenfunctions that form an orthonormal basis of $L^2(\Omega)$, we can define in $\mathcal{H}$ an inner product similarly as in \cite{CaCeGa}.

Notice that we can also find a representation by Fourier series for the solutions of the system \eqref{linear} (see e.g. \cite{JaKo}, for one-dimensional Schrödinger equation) and that the operator $\mc A$, defined by \eqref{op-A}-\eqref{dom-A}, is skew-adjoint and hence generates an infinitesimal generator of a group $\{S(t)\}_{t\in \rea}$, thus (H1) follows. Also note that $g(-\mc A)=0$, where $g$ is the growth bound of the semigroup generated by $\mc A$. Moreover, comparing the abstract system \eqref{eq:abs} with our system \eqref{linear}, the control operator $\mc B \in \mc L(\cplx; \mc D(\mc A^*)')$ can be given as follows
$$
	\ip{\mc B s}{ (\phi_1,\phi_2)}_{\mc D(\mc A^*)', \mc D(\mc A^*)}=s \phi_2'(0), \ \ s\in \cplx, \ (\phi_1,\phi_2)\in \mc D (\mc A^*),
$$
and therefore, (H2) is verified.  Additionally, note that the observability inequality \eqref{Obs} gives directly (H4).

It remains for us to prove the hypothesis (H3), that is, to prove the trace regularity.  This hypothesis follows from the next proposition and is a consequence of Lemma \ref{id_lm}. 
\begin{proposition}\label{irena1}
	For every $T>0$ there exists $C>0$ such that the following holds
	\begin{equation}\label{admis}
		\int_{0}^{T}|\vphi_{2,x}(t,0)|^2dt \leq C \norm{(\zeta_1, \zeta_2)}_{\mathcal{H}},
	\end{equation}
	for every solution $(\vphi_1, \vphi_2)$ of the adjoint problem \eqref{adjoint} with $(\zeta_1, \zeta_2)$ lies in sufficiently regular space.
\end{proposition}
\begin{proof}
Recall that $\Omega=(0,1)$ and choose $m\in C^2(\Omega)$ such that $m(1)=0, m(0)>0$, and   $m'(1)=1$. Thus, from \eqref{tr_identity}, it follows that
\begin{equation}\label{pr tr}
\begin{split}
\frac{1}{2}\sum_{j=1}^{2}\nu_j&\int_{0}^{T}|\vphi_{j,x}(t,0)|^2m(0) dt\leq C_1\norm{m}_{L^{\infty}(\Omega)}\bigg[\norm{(\vphi_1(T),\vphi_2(T))}^2_{[L^2(\Omega)]^2}\\
&+\norm{(\vphi_1(T),\vphi_2(T))}^2_{\mathcal{H}}+\norm{(\vphi_1(0),\vphi_2(0))}^2_{[L^2(\Omega)]^2}+\norm{(\vphi_1(0),\vphi_2(0))}^2_{\mathcal{H}}\bigg]\\&+C_2\norm{m}_{W^{1,\infty}(\Omega)}\left(\int_{0}^{T}\norm{(\vphi_1(t),\vphi_2(t))}^2_{\mathcal{H}}dt\right)\\
	&+C_3\norm{m}_{W^{2,\infty}(\Omega)}\left(\int_{0}^{T}\norm{(\vphi_1(t),\vphi_2(t))}_{[L^2(\Omega)]^2}\norm{(\vphi_1(t),\vphi_2(t))}_{\mathcal{H}}dt\right).
\end{split}
\end{equation}
Note that, using boundary conditions of \eqref{adjoint}, one can prove that the last term of the identity \eqref{tr_identity} is negative.
Moreover, using the classical conservation of the energy for the Schr\"odinger equation we have:
\begin{align}
\label{1 l2}	&\norm{(\vphi_1(t), \vphi_2(t))}_{[L^2(\Omega)]^2}=\norm{(\vphi_1(0), \vphi_2(0))}_{[L^2(\Omega)]^2}, \ \  \forall t\in [0,T],\\
\label{2 h1}	&\norm{(\vphi_1(t), \vphi_2(t))}_{\mathcal{H}}=\norm{(\vphi_1(0), \vphi_2(0))}_{\mathcal{H}}, \ \ \forall t\in [0,T].
\end{align}

Therefore from \eqref{pr tr}, we have the following result \eqref{admis} when $(\zeta_1, \zeta_2)$ is sufficiently smooth. Using a density argument we get the result when $(\zeta_1, \zeta_2)\in \mathcal{H}$, showing the lemma, and consequently (H3) is achieved.
\end{proof}

\subsection{Proof of Theorem \ref{th-main-1}}
In this section, we employ Urquiza's method to construct the feedback law for our system \eqref{linear} to ensure the exponential decay \eqref{ex_p}.

To do that, let us consider $ \Phi^1_0=(\vphi_0, \psi_0)$ and  $\Phi^2_0=(r_0, s_0)$ belonging of $\mathcal{H}$, and define the following bilinear form
$$a_w(\Phi^1_0, \Phi^2_0)=\int_{0}^{\infty}e^{-2\omega t}\psi_{x}(t,0) \overline{s_{x}(t,0) }\, dt,$$
where $\Phi^1=(\vphi, \psi)$ and $\Phi^2=(r, s)$ are the solutions of the following
systems respectively
\begin{equation*}
	\begin{cases}
		i\vphi_{t}+ \gamma_1\vphi_{xx} -\alpha_1\vphi = 0, &\text{ in } Q_T,\\
		i\sigma \psi_{t} +  \gamma_2\psi_{xx}-\alpha_2 \psi= 0, & \text{ in } Q_T,\\
		\vphi(t,1)=\psi(t,1), & \text{ in } (0,T),\\
		\gamma_1\vphi_{x}(t,1)+\frac{\gamma_2}{\sigma} \psi_{x}(t,1)+\alpha \vphi(t,1)=0, & \text{ in } (0,T),\\
		\vphi(t,0)=0,\, \psi(t,0)=0, & \text{ in } (0,T),\\
		\vphi(T,x)=\vphi_0(x), \  \psi(T,x)=\psi_0(x), & \text{ in } \Omega
	\end{cases} 
\end{equation*} 
and
\begin{equation*}
	\begin{cases}
		ir_{t}+ \gamma_1 r_{xx} -\alpha_1 r= 0, &\text{ in } Q_T,\\
		i\sigma s_{t} +  \gamma_2 s_{xx}-\alpha_2 s= 0, & \text{ in } Q_T,\\
		r(t,1)=s(t,1), & \text{ in } (0,T),\\
		\gamma_1 r_{x}(t,1)+\frac{\gamma_2}{\sigma} s_{x}(t,1)+\alpha r(t,1)=0, & \text{ in } (0,T),\\
		r(t,0)=0,\, s(t,0)=0, & \text{ in } (0,T),\\
		r(T,x)=r_0(x), \  s(T,x)=s_0(x), & \text{ in } \Omega. 
	\end{cases} 
\end{equation*}
Let us define the operator $\Lambda_{\omega}:\mathcal{H}\to \mathcal{H}'$ satisfying the following
\begin{equation}\label{lambda}
	\ip{\Lambda_{\omega} \Phi^1_0 } {\Phi^2_0}_{\mathcal{H}',\mathcal{H}}=a_{\omega}(\Phi^1_0 ,\Phi^2_0), \quad \forall  \Phi^1_0,\Phi^2_0 \in \mathcal{H}.
\end{equation}
Next, we see that
\begin{equation*}
\begin{split}
 a_w(\Phi^1_0 ,\Phi^2_0)=&\int_{0}^{\infty}e^{-2\omega t}\psi_{x}(t,0) \overline{s_{x}(t,0) }\,dt\\
=&\int_{0}^{\infty}e^{-2\omega t}{\mc B^*} \Phi^1(t) \,\,  \overline{{\mc B^*} \Phi^2(t)} \, dt\\
	=&\int_{0}^{\infty}e^{-2\omega t}\left({\mc B^*} S(T-t)^* S(-T)^*\Phi^1_0\right) \overline{\left({\mc B^*} S(T-t)^* S(-T)^*\Phi^2_0\right)}\, dt\\
	=&\int_{0}^{\infty}e^{-2\omega t}\left({\mc B^*}S(-t)^*\Phi^1_0\right) \overline{\left({\mc B^*} S(-t)^* \Phi^2_0\right)}\, dt.
	\end{split}
\end{equation*}
Therefore from \eqref{lambda}, we have
$$
	\ip{\Lambda_{\omega} \Phi^1_0 } {\Phi^2_0}_{\mathcal{H}',\mathcal{H}}=\int_{0}^{\infty}e^{-2\omega t}\ip{{\mc B^*}S(-t)^*\Phi^1_0}{\, {\mc B^*}S(-t)^* \Phi^2_0}_{\cplx} \, dt.
$$
Thanks to Theorem \ref{thm}, the operator $\Lambda_{\omega}$, defined by \eqref{lambda}, is coercive and isomorphism. Finally, let us define the operator ${F}_{\omega}: \mathcal{H} \to \mathbb{C}$ by
\begin{align*}
	{F}_{\omega}(\mathbf{z})=-\psi'_0(0) ,	\end{align*} where $\Phi^1_0=(\vphi_0, \psi_0)$ is the solution of the following Lax-Milgram  problem
\begin{equation*}
	a_{\omega}(\Phi^1_0, \Phi^2_0)=\ip{\mathbf{z}}{\Phi^2_0}, \ \forall \, \Phi^2_0 \in \mathcal{H}.
\end{equation*}
Hence, we obtain $$\ip{\Lambda_{\omega} \Phi^1_0 } {\Phi^2_0}_{\mathcal{H}',\mathcal{H}}=\ip{\mathbf{z}}{\Phi^2_0},  \ \forall \, \Phi^2_0 \in \mathcal{H}.$$ This gives $\Lambda_{\omega}\Phi^1_0=\mathbf{z}.$ It follows that $\Phi^1_0=\Lambda_{\omega}^{-1} \mathbf{z}.$ Thus we have ${F}_{\omega}=-{\mc B}^*\Lambda_{\omega}^{-1}.$
Thanks to Theorem \ref{thm}, rapid exponential stabilization for the system \eqref{linear} is established using the feedback law $h(t)={F}_{\omega}(u(t,\cdot), v(t,\cdot))$. More precisely, we get a positive constant $C$ such that the solution of \eqref{linear} satisfies the estimate \eqref{ex_p}, showing Theorem \ref{th-main-1}. \qed

\begin{remark}
We end this section with the following comments. 
\begin{itemize}
\item[i.] As mentioned before, an estimate similar to \eqref{admis} implies the well-posedness of the control system  \eqref{linearsource}. We infer the reader to see Proposition \ref{wl th}.
\item[ii.] Estimate \eqref{admis} and exact controllability (Theorem \ref{th-main}) together with the hypothesis (H2) ensures the uniform stabilizability property considering the dissipative feedback $u(t) =-\mc B^*y(t)$, where $\mc B^*$ is the adjoint control operator of $\mc B$ and $y = (u, v)$ is the solution of the closed-loop system \eqref{linear} with boundary conditions \eqref{boundary-1}-\eqref{control-v}. A very elegant and short proof of this fact can be found in \cite[Proposition 3.1]{Irena}.
\end{itemize}
\end{remark}


\section{Final comments and open problems}\label{sec4}

In this paper, we considered the boundary control problems for coupled Schrödinger equations through the Kirchhoff boundary conditions in a one-dimensional case. The first result is obtained showing a new Carleman estimate with boundary observation. Moreover, with this in hand, together with other hypotheses over the operator $\mathcal{A}$, the second result ensures that the solutions of the system decay exponentially with a decay rate of at least $e^{-2\omega t}$.

Concerning Theorems \ref{th-main} and \ref{th-main-1} it is natural to ask whether they remain valid in nonlinear framework \eqref{nonlinear}. Due to the lack of regularity, we are not able to extend our result for the nonlinear case yet, and this issue remains open.

Finally, we point out that in a forthcoming paper boundary conditions different of \eqref{boundary-1}-\eqref{control-u}/ \eqref{control-v} will be considered, which will give a more general view of the boundary control problems, at least, for the system \eqref{linear}.

	
	\section*{Statements and Declarations}

\subsection*{Acknowledgments} The authors are grateful to the anonymous reviewer for his/her constructive comments and valuable remarks.

The work of Kuntal Bhandari was supported by the Czech-Korean project  GA\v{C}R/22-08633J.  
 Capistrano-Filho was supported by CAPES grants numbers \linebreak 88881.311964/2018-01 and 88881.520205/2020-01,  CNPq grants numbers  307808/2021-1 and  \linebreak 401003/2022-1,  MATHAMSUD grant 21-MATH-03 and Propesqi (UFPE). Subrata Majumdar received financial support from the institute post-doctoral fellowship of IIT Bombay.  Part of this work was done while the second author visited Virginia Tech from January to July 2023 and Centro de Modelamiento Matemático (Santiago-Chile) in October 2023. The author thanks both institutions for their warm hospitality. 

\subsection*{Authors contributions:} Bhandari, Capistrano-Filho, Majumdar, and Tanaka work equality in Conceptualization; formal analysis; investigation; writing--original draft; writing--review and editing.

\subsection*{Conflict of interest statement} This work does not have any conflicts of interest.

\subsection*{Data availability statement} It does not apply to this article as no new data were created or analyzed in this study.


\appendix
	\section{Auxiliary results}\label{apx}	
	In this first appendix, we briefly discuss the well-posedness of the control system  \eqref{linear}.

	\subsection{Well-posedness results}	
	Consider the following  operator associated with the control system \eqref{linear}, given by 
	\begin{align}\label{op-A}
		\mc A=\begin{pmatrix}
			i \gamma_1 \partial_{xx}  - i\alpha_1 \mathbb I_d & \mathbf{0}\\ \\
			\mathbf{0}& i\frac{\gamma_2}{\sigma} \partial_{xx} - i \frac{\alpha_2}{\sigma}  \mathbb I_d
		\end{pmatrix}    ,
	\end{align} 
	with 
	\begin{equation}\label{dom-A}
	\begin{split}
		\mc D(\mc A)=\big\{(u_1, u_2)\in [H^2(\Omega)]^2 \mid \, & u_1(0)=u_2(0)=0, \  u_1(1)=u_2(1), \\
		&  \gamma_1 u_{1,x}(1)+\frac{\gamma_2}{\sigma}u_{2,x}(1)+\alpha u_1(1)=0 \big\}.
	\end{split}
	\end{equation}
	%
	%
	
	With this in hand, the first result shows that the operator in consideration is dissipative. 
	
	\begin{proposition}
		The operator $(\mc A, \mc{D}(\mc A))$ generates a strongly continuous unitary group in $[L^2(\Omega)]^2$.
	\end{proposition}
	\begin{proof}
		Let us consider $\mathbf U=(u,v) \in \mathcal D(\mathcal A)$. A simple computation gives that
				\begin{equation*}
				\begin{split}
		  \ip{\mc A\mathbf U}{\mathbf U}_{[L^2(\Omega)]^2}=&\re\bigg[i\gamma_1\int_{0}^{1}u_{xx}\overline{u}dx-i\alpha_1\int_{0}^{1}|u|^2dx+i\frac{\gamma_2}{\sigma}\int_{0}^{1}v_{xx}\overline{v}dx-i\frac{\alpha_2}{\sigma}\int_{0}^{1}|u|^2dx\bigg]\\
			=&\re\bigg[-i\gamma_1\int_{0}^{1}|u_{x}|^2dx-i\frac{\gamma_2}{\sigma}\int_{0}^{1}|v_{x}|^2dx+i\gamma_1u_x(1)\overline{u(1)}+i\frac{\gamma_2}{\sigma}v_x(1)\overline{v(1)}\bigg]\\
			=&\re\bigg[i\gamma_1u_x(1)\overline{u(1)}+i\frac{\gamma_2}{\sigma}v_x(1)\overline{v(1)}\bigg]\\=&\re\bigg[-i\alpha |u(1)|^2\bigg]=0.
			\end{split}
		\end{equation*}
		By using semigroup theory, $\mc A$ generates a strongly continuous unitary group on $[L^2(\Omega)]^2.$
		Moreover it can be easily checked that for all $(u,v)\in \mc D(\mc A)$, 
		$$\ip{\mc A(u,v)}{(u,v)}_{[L^2(\Omega)]^2}=-\ip{(u,v)}{\mc A(u,v)}_{[L^2(\Omega)]^2}$$ and also  
		$\mc D(\mc A)=\mc D(\mc A^*).$ Therefore $(\mc A, \mc D(\mc A))$ is a skew adjoint operator.
	\end{proof}
	
	%
	%
	
	\subsection{Adjoint system}
	Let us remember that the adjoint associated with the system \eqref{linear} with \eqref{boundary-1}--\eqref{control-u} or \eqref{boundary-1}--\eqref{control-v}, is given by \eqref{adjoint}, with given final data $\zeta:=(\zeta_1, \zeta_2)$ from some suitable Hilbert space. Note that, since $\mathcal{A}=-\mathcal{A}^*$ we have that $\mathcal{D(A^*)} = \mathcal{D(A)}$, so, the following result shows the well-posedness for the system \eqref{adjoint}.
	\begin{proposition}
		For given $\zeta:=(\zeta_1, \zeta_2)\in \mathcal H$, there exists a unique solution $\vphi:=(\vphi_1, \vphi_2)\in C([0,T]; \mc H)$ to the adjoint system \eqref{adjoint}  such that it satisfies
		\begin{align}\label{esti-adj}
			\|\vphi\|_{C([0,T];\mathcal H)} \leq C \|\zeta\|_{\mathcal H} ,
		\end{align}
		for some constant $C>0$. 
	\end{proposition}

\section{Key lemma}\label{apx2}	
This second part of this appendix is devoted to presenting an essential lemma that is a key point in ensuring the hypothesis (H3) in Urquiza's approach.
\begin{lemma}\label{id_lm}
	First, consider a function $m \in C^2(0,1)$. Then the solution of \eqref{adjoint} satisfies the following identity:
	\begin{equation}\label{tr_identity}
\begin{split}
	&\frac{1}{2}\sum_{j=1}^{2}\nu_j\int_{0}^{T}|\vphi_{j,x}(t,0)|^2m(0) dt=\frac{1}{2}\sum_{j=1}^{2}\nu_j\int_{0}^{T}|\vphi_{j,x}(t,1)|^2m(1) dt\\&-\frac{1}{2}\im\left(\sum_{j=1}^{2} \int_{0}^{1} \br{\vphi_{j}(T,x)}m(x)\vphi_{j,x}(T,x)dx\right)+\frac{1}{2}\im\left(\sum_{j=1}^{2} \int_{0}^{1} \br{\vphi_{j}(0,x)}m(x)\vphi_{j,x}(0,x)dx\right)\\
&+\frac{1}{2}\im\left(\sum_{j=1}^{2} \int_{0}^{T} \vphi_{j,t}(t,1)m(1)\br{\vphi_{j}(t,1)}dt\right)-\frac{1}{2}\re\left(\sum_{j=1}^{2} \nu_j \int_{0}^{T}\int_{0}^{1} \vphi_{j,x} m'(x)\br{\vphi_{j,x}(t,x)}dxdt\right)\\
&-\frac{1}{2}\re\left(\sum_{j=1}^{2} \nu_j \int_{0}^{T}\int_{0}^{1} \vphi_{j,x} m''(x)\br{\vphi_{j}(t,x)}dxdt\right)+\frac{1}{2}\re\left(\sum_{j=1}^{2} \nu_j \int_{0}^{T} \vphi_{j,x}(t,1) m'(1)\br{\vphi_{j}(t,1)}dt\right),\\
		&-\frac{1}{2}\re\left(\sum_{j=1}^{2} \theta_j \int_{0}^{T} m(1)|{\vphi_{j}(t,1)}|^2dt\right),
		\end{split}
	\end{equation}
	where $\nu_1=\gamma_1, \nu_2=\frac{\gamma_2}{\sigma}, \theta_1=\alpha_1, \theta_2=\frac{\alpha_2}{\sigma}$.
\end{lemma}
\begin{proof}
Multiply the equations \eqref{adjoint} by $(m\br{\vphi_{1,x}}+\frac{1}{2}\br{\vphi_1}m')$ and $(m\br{\vphi_{1,x}}+\frac{1}{2}\br{\vphi_1}m')$ respectively and using integration by parts along with boundary conditions, we have
\begin{equation*}
\begin{split}
\re\sum_{j=1}^{2} i \int_{0}^{T}&\int_{0}^{1} \vphi_{j,t}\left( \frac{1}{2}m'(x)\br{\vphi_{j}(t,x)}+m(x) \br{\varphi_{j,x}(t,x)}\right)dxdt=\\&-\frac{1}{2}\im\left(\sum_{j=1}^{2} \int_{0}^{1} \br{\vphi_{j}(T,x)}m(x)\vphi_{j,x}(T,x)dx\right)\\
&+\frac{1}{2}\im\left(\sum_{j=1}^{2} \int_{0}^{1} \br{\vphi_{j}(0,x)}m(x)\vphi_{j,x}(0,x)dx\right)\\&+\frac{1}{2}\im\left(\sum_{j=1}^{2} \int_{0}^{1} \vphi_{j,t}(t,1)m(1)\br{\vphi_{j}(t,1)}dt\right).
	\end{split}
\end{equation*}
Similarly, we have from the second term of both equations:
\begin{equation*}
\begin{split}
	\re\sum_{j=1}^{2} \nu_j \int_{0}^{T}\int_{0}^{1} \vphi_{j,xx}&\left( \frac{1}{2}m'(x)\br{\vphi_{j}(t,x)}+m(x) \br{\varphi_{j,x}(t,x)}\right)dxdt=\frac{1}{2}\sum_{j=1}^{2}\nu_j\int_{0}^{T}|\vphi_{j,x}(t,1)|^2m(1) dt\\
	&-\frac{1}{2}\re\left(\sum_{j=1}^{2} \nu_j \int_{0}^{T}\int_{0}^{1} \vphi_{j,x} m'(x)\br{\vphi_{j,x}(t,x)}dxdt\right)\\&-\frac{1}{2}\re\left(\sum_{j=1}^{2} \nu_j \int_{0}^{T}\int_{0}^{1} \vphi_{j,x} m''(x)\br{\vphi_{j}(t,x)}dxdt\right)\\
	&+\frac{1}{2}\re\left(\sum_{j=1}^{2} \nu_j \int_{0}^{T} \vphi_{j,x}(t,1) m'(1)\br{\vphi_{j}(t,1)}dt\right)\\
	&-\frac{1}{2}\left(\sum_{j=1}^{2} \nu_j \int_{0}^{T} m(0)|{\vphi_{j,x}(t,0)}|^2dt\right).
		\end{split}
\end{equation*}
Also, we have 
$$\re\sum_{j=1}^{2} \theta_j \int_{0}^{T}\int_{0}^{1} \vphi_{j}\left( \frac{1}{2}m'(x)\br{\vphi_{j}(t,x)}+m(x) \br{\varphi_{j,x}(t,x)}\right)dxdt=\left(\sum_{j=1}^{2} \theta_j \int_{0}^{T} m(1)|{\vphi_{j}(t,1)}|^2dt\right).$$
Putting together the previous relation \eqref{tr_identity} holds and the lemma is finished.
\end{proof}

\end{document}